\documentclass[reqno, fleqn, 10pt]{amsart}

\usepackage[english]{babel}

\usepackage{mathtools}
\usepackage{color}
\usepackage{comment}
\usepackage{amsmath}
\usepackage{amssymb}
\usepackage{amsfonts}
\usepackage{amsthm}
\usepackage{mathrsfs}
\usepackage[utf8x]{inputenc}

\usepackage[autostyle]{csquotes}

\newcommand{\N}{\mathbb{N}}

\newcommand{\Z}{\mathbb{Z}}

\newcommand{\f}{\rightarrow}

\newcommand{\bdr}{\partial}
\newcommand{\fix}{\mathsf{Fix}}
\newcommand{\orientable}[1]{\mathord{\mathop{#1}\limits^{\scriptscriptstyle(\sim)}}}

\newcommand{\ortimes}{\orientable{\times}}
\newcommand{\Stab}{\operatorname{Stab}}

\newtheorem{thm}{Theorem}[section]
\newtheorem*{thm*}{Theorem}
\newtheorem*{JSJdecthm}{JSJ-decomposition Theorem}
\newtheorem*{hyperbolization}{Hyperbolization Theorem}

\newtheorem{prop}[thm]{Proposition}
\newtheorem{lem}[thm]{Lemma}

\theoremstyle{remark}

\newtheorem{rmk}[thm]{Remark}

\theoremstyle{definition}
\newtheorem{defn}[thm]{Definition}

\title[On the peripheral subgroups of $3$-manifold groups and splittings]{On the peripheral subgroups of irreducible $3$-manifold groups and acylindrical splittings}
\author[F. Cerocchi]{Filippo Cerocchi}
\email{fcerocchi@mpim-bonn.mpg.de; fcerocchi@gmail.com}
\address{Max-Planck Institut f\"ur Mathematik, Vivatsgasse 7, 53111, Bonn, Germany. \newline}
\thanks{The author wish to thank the Max-Planck-Institut f\"ur Mathematik for the  financial support and the excellent working conditions.\\}
\begin{document}

\maketitle

\vspace{-3mm}
{\bf Abstract.} We discuss the problem of separation under conjugacy and malnormality of the abelian peripheral subgroups of an orientable, irreducible $3$-manifold $X$. We shall focus on the relation between this problem and the existence of acylindrical splittings of $\pi_1(X)$ as an amalgamated product or  HNN-extension along the abelian subgroups corresponding to the JSJ-tori, providing new proofs of results in \cite{wilton2010profinite} and \cite{delaHarpe2014onmalnormal}.

\section{Introduction}

%In this paper we  shall be concerned with the peripheral subgroups associated to the boundary tori  of orientable, compact, irreducible $3$-manifolds. \\
%We recall that a compact $3$-manifold $X$ is said to be \textit{irreducible} if every embedded $2$-sphere in $X$ bounds a $3$-ball.  Notice that irreducible $3$-manifolds are prime (\textit{i.e.} $X\simeq X_1\# X_2$ implies that either $X_1$ or $X_2$ is homeomorphic to $S^3$). Moreover, irreducible $3$-manifolds with boundary which are not homeomorphic to the $3$-ball do not admit spherical boundary components. \\

%\noindent {\bf Standing assumption.} In the sequel all $3$-manifolds are assumed to be connected, compact and orientable (unless otherwise stated).\\

A classical result proved in the late seventies by Jaco-Shalen and independently by Johannson (see \cite{jaco1979seifert}, \cite{johannson1979homotopy}) says that an orientable, irreducible $3$-manifold $X$ with (possibly empty) boundary can be splitted along a minimal collection of embedded, \textit{incompressible}\footnote{We say that a compact surface $S$ properly embedded in a compact $3$-manifold $X$ is incompressible if for any embedded disk $D$ such that $\bdr D\subset S$, there exists an embedded disk $D'\subset S$  such that $\bdr D'=\bdr D$.} tori which are not \textit{boundary parallel}\footnote{A properly embedded compact submanifold $Y$ of a compact manifold $X$ is said to be boundary parallel if it can be isotoped into  $\bdr X$ by an isotopy fixing $\bdr Y\subset \bdr X$.} and which decompose the manifold $X$ into connected components which are either \textit{atoroidal} (\textit{i.e.} any embedded, incompressible torus  is boundary parallel) or \textit{Seifert fibered} (a  well understood class of $3$-manifolds, see subsection \S2.1). Such a collection  is unique up to isotopy, and is called the collection of the \textit{JSJ-tori} of $X$. The decomposition along this collection of tori is called  {\it JSJ-decomposition} and the connected components obtained via the previous procedure are the {\it JSJ-components} of $X$. The set of atoroidal irreducible $3$-manifolds and the set of Seifert fibered $3$-manifolds have a non empty intersection: actually, atoroidal, Seifert fibered $3$-manifolds are classified (see \cite{jaco1979seifert}, \S IV.2.5 and \S IV.2.6).
 By Thurston's Hyperbolization Theorem 
 %(see \cite{thursto1982kleinian}, \cite{thurston1986hyperbolic}, \cite{thurston1998hyperbolicII}, \cite{thurston1998hyperbolicIII}, and  \cite{otal1996theorem}, \cite{otal1998thurston} and \cite{kapovich2001hyperbolic} for complete proofs) states that the interior of a compact, irreducible Haken $3$-manifold admits a complete hyperbolic metric if and only if the manifold is homotopically atoroidal and not homeomorphic to $K\widetilde{\times} I$. The condition of being homotopically atoroidal is slightly stronger than the condition of being atoroidal (see section \S2.2), but the two notions coincide except for a list of Seifert fibered manifolds. In particular the result holds for compact, irreducible $3$-manifolds with non-empty boundary. Thanks to 
and by the work of Perelman, the interior of an atoroidal, non-Seifert fibered, irreducible $3$-manifold can be endowed with a complete hyperbolic metric which has finite volume if and only if one of the following holds: either the manifold is closed or all of its boundary components are  homeomorphic to tori.
We shall refer to the atoroidal, non-Seifert fibered, irreducible $3$-manifolds as to the \textit{manifolds of hyperbolic type}.
We refer to subsection \S\ref{JSJDEC} for further details and references.\\

%\noindent \textit{Notation.} Given a manifold $X$ we implicitely assume to have fixed a base point $x\in X$ and we  denote simply $\pi_1(X)$ its fundamental group.\\ Let $X$ be a compact manifold,  $Y\subset X$ a submanifold, and let $\iota: Y\hookrightarrow X$ be the inclusion map of $Y$ into $X$; we shall denote by $\pi_1(Y)$ the subgroup of $\pi_1(X)$ given by $\iota_*(\pi_1(Y))$, \textit{i.e.} we assume to have fixed both base points $x\in X$ and $y\in Y$, and  a path from $x$ to $\iota(y)$. Let $X$ be a manifold, $S$ be a connected component in $\bdr X$, and $Y$ a submanifold whose inclusion in $X$ has already been fixed, we choose $T\hookrightarrow X$ so that we have the factorization  $T\hookrightarrow Y\hookrightarrow X$. \\

We are interested in the study of the \textit{peripheral subgroups} corresponding to the boundary tori of an orientable, irreducible $3$-manifold $X$, \textit{i.e.} the \textit{abelian peripheral subgroups of }$\pi_1(X)$. The \textit{peripheral subgroups} of the fundamental group of a given compact $3$-manifold $X$ are the subgroups of $\pi_1(X)$ which are images via the natural inclusions\footnote{Notice that the \textit{peripheral subgroups} are defined up to the choice of  basepoints $x_0$ in $X$, $x_i$ in $\bdr_i X$ and of a path connecting  $x_0$ to $x_i$.} of the fundamental groups of the connected components of $\bdr X$.
In particular we shall focus on the following question: how do the abelian peripheral subgroups of  $\pi_1(X)$ behave under conjugacy in $\pi_1(X)$?\\%By the term \textit{non-geometric} we mean a compact $3$-manifold $X$ such that neither $X$ nor $int(X)$ (the interior of the manifold $X$) admit a finite volume metric which is locally isometric to one of the eight model geometries (see \cite{thurston1997three}, \cite{scott1983geometries}). 
%\begin{itemize}
%\item[(Q1)] Let $T^{\bdr}$ be a boundary torus of an irreducible $3$-manifold $X$. Is the peripheral subgroup $\pi_1(T^{\bdr})$ \textit{malnormal}\footnote{A subgroup $H<G$ is said to be \textit{malnormal} in $G$ if $g H g^{-1}\cap H=\{1\}$ for any $g\in G\smallsetminus H$.} in $\pi_1(X)$? If it is not, what can be said about the intersection $g\,\pi_1(T^{\bdr})\,g^{-1}\cap\pi_1(T^\bdr)$ as $g$ varies in $\pi_1(X)$?
%\item[(Q2)] Let $T_{ 1}^{\bdr}$ and $T_{2}^{\bdr}$ be two boundary tori of a given irreducible $3$-manifold $X$. Are the corresponding peripheral subgroups  conjugately separated? If they are not, what can be said about the intersection $g\pi_1(T_{1}^{\bdr})g^{-1}\cap\pi_1(T_{2}^{\bdr})$ as $g$ varies in $\pi_1(X)$?
%\end{itemize}

 In  \cite{delaHarpe2014onmalnormal} de la Harpe and Weber proved the following:
 \begin{thm*}[\cite{delaHarpe2014onmalnormal}, Theorem 3]
 	Let $X$ be an orientable, irreducible $3$-manifold $X$ and let $T\subset \bdr X$ be a boundary torus. The peripheral subgroup $\pi_1(T)$ is malnormal\footnote{A subgroup $H<G$ is said to be \textit{malnormal} in $G$ if $g H g^{-1}\cap H=\{1\}$ for any $g\in G\smallsetminus H$.} in $\pi_1(X)$ if and only if  $T$ bounds a JSJ-component of hyperbolic type.
 \end{thm*}
 
   The malnormality of a suitable family of subgroups may have interesting consequences. As an example, the previous result can be used to characterize those $3$-manifold groups which are CA or CSA\footnote{A CA group is a group $G$ such that for each element $g\in G$ the centralizer $C(g)$ is abelian; a CSA group is a group whose maximal abelian subgroups are malnormal.} (see Corollary 2.5.11 in \cite{aschenbrenner20153}).\\

 Given an orientable, irreducible $3$-manifold $X$, we shall focus instead on the boundary tori of $X$ lying into Seifert fibered JSJ-components of $X$. The next result, together with Theorem 3 in \cite{delaHarpe2014onmalnormal}, gives a picture of the cases that can occur (except for $D^2\times S^1$, $T^2\times I$, $K\widetilde{\times}I$ --- the twisted interval bundle over the Klein bottle --- where the answer is trivial):

\begin{thm}\label{AI_bdrtori}
Let $X$ be an orientable, irreducible $3$-manifold with non-empty boundary and assume that $X$ is not homeomorphic to  $D^2\times S^1$, $T^2\times I$, $K\widetilde\times I$. Let $T_{(k_1,i_1)}^{\bdr}$, $T_{(k_2,i_2)}^{\bdr}$ denote two (possibly equal) boundary tori of $X$ contained into two (possibly equal) Seifert fibered JSJ-components $X_{k_1},\, X_{k_2}$ of $X$. 
\begin{itemize}
	\item[(i)]  Assume that \small$(k_1,i_1)=(k_2,i_2)=(k,i)$\normalsize\, and let \small$g\in\pi_1(X)\smallsetminus\pi_1(T_{(k,i)}^{\bdr})$\normalsize. Then
	\small
	 $$g\pi_1(T_{(k,i)}^{\bdr})g^{-1}\cap\pi_1(T_{(k,i)}^{\bdr})\neq\{1\}$$
	 \normalsize
	  if and only if \small$g\in\pi_1(X_{k})$\normalsize. In  this case\, \small$g\pi_1(T_{(k,i)}^{\bdr})g^{-1}\cap\pi_1(T_{(k,i)}^{\bdr})=\langle f_k\rangle$\normalsize,\, where \small$f_k$\normalsize\, is a regular fiber of \small$\pi_1(X_k)$\normalsize.
	\item[(ii)] Assume that \small$(k_1,i_1)\neq (k_2,i_2)$\normalsize,\, and let \small$g\in\pi_1(X)$\normalsize. Then
	\small
	$$g\pi_1(T_{(k_1,i_1)}^{\bdr})g^{-1}\cap\pi_1(T_{(k_2,i_2)}^{\partial})\neq\{1\}$$
	\normalsize
	if and only if \small$k_1=k_2=k$\normalsize,\, and  \small$g\in\pi_1(X_k)$\normalsize. In this case we have\linebreak \small$g\pi_1(T_{(k, i_1)}^{\bdr})g^{-1}\cap\pi_1(T_{(k, i_2)}^{\bdr})=\langle f_k\rangle$\normalsize,\, where \small$f_k$\normalsize\, is a regular fiber of \small$\pi_1(X_k)$\normalsize.
\end{itemize}
\end{thm}

\noindent By a \textit{regular fiber} of $\pi_1(X_k)$ we mean an element representing the homotopy class of the regular fiber of a Seifert fibration on $X_k$. It turns out (see Remark \ref{regular_fiber}) that, in restriction to the cases which can appear in Theorem \ref{AI_bdrtori}, the elements representing the homotopy classes of the regular fibers of the different Seifert fibrations of $X_k$ coincide (up to take inverses). This provides a univoquely determined \textit{regular fiber of $\pi_1(X_k)$}, which depends only on $X_k$.

\begin{rmk}
	Theorem \ref{AI_bdrtori} and Theorem 3 in \cite{delaHarpe2014onmalnormal} will appear as consequences of Proposition \ref{int_leaves_stab} in this paper, where we study the intersection of the leaves stabilizers of the Bass-Serre tree of a suitable graph of groups associated to the manifold $X$ (see section \S3).
\end{rmk}

\begin{rmk}\label{basepoints}
	Notice that the $3$-manifolds excluded by the previous statements are precisely the only three orientable,  Seifert fibered $3$-manifolds with non-empty boundary having virtually abelian fundamental group.
\end{rmk}

We  shall use Theorem \ref{AI_bdrtori}  to give a proof of the existence of a particular class of splittings of the fundamental group of an orientable, irreducible $3$-manifold $X$. In order to state the result we shall briefly recall some facts about $k$-step malnormal amalgamated products, $k$-step malnormal HNN-extensions and $k$-acylindrical splittings. \\

\noindent In 1971 Karrass-Solitar (see \cite{karras1971malnormal}) introduced the notion of \textit{$k$-step malnormal amalgamated product}. Roughly speaking, a $k$-step malnormal amalgamated product $A\,^*_CB$ is an amalgamated product where the amalgam $C$ is malnormal with respect to the elements having \textit{syllable length} greater or equal to $k+1$. In analogy we can define a \textit{$k$-step malnormal HNN-extension} as a HNN-extension $A^*_{\varphi}$, with respect to an isomorphism of subgroups $\varphi :C_{-1}\f C_1$, such that the associated subgroups,  $C_{\pm 1}$, are malnormal and conjugately separated in $A\,_\varphi^{*}$ with respect to the set of elements having a reduced form containing at least $k+1$ times the stable letter $t$ to the power $\pm 1$ (see subsection \S4.1.1 for a precise definition).\\ These notions have a natural generalization in the context of Bass-Serre theory:  the notion of $k$-acylindrical splitting introduced by Sela in \cite{sela1997acylindrical}. We recall that a finitely generated group $G$ admits a \textit{$k$-acylindrical splitting} if there exists a graph of groups $(\mathscr G,\Gamma)$ such that $G=\pi_1(\mathscr G, \Gamma)$ and the action of $G$ on the corresponding Bass-Serre tree $\mathcal T_{(\mathscr G, \Gamma)}$ is $k$-acylindrical, \textit{i.e.} the diameter of the fixed points set of any element of $G$  is smaller or equal to $k$. It turns out  that a splitting of $G$ as a $k$-step malnormal amalgamated product gives a $(k+1)$-acylindrical splitting and a splitting of $G$ as a $k$-step malnormal HNN-extension is a $k$-acylindrical splitting. Conversely, a $k$-acylindrical splitting whose underlying graph has only one edge with two distinct vertices corresponds to a splitting as a $(k-1)$-step malnormal amalgamated product and a $k$-acylindrical splitting having a loop as underlying graph corresponds to a $k$-step malnormal HNN-extension. We shall explain this correspondence in subsection \S4.1.2.

 The existence of acylindrical splittings for the fundamental groups of closed, orientable, irreducible $3$-manifolds with non-trivial JSJ-decomposition, which are not $Sol$-manifolds, has been proved by Wilton-Zalesskii (see \cite{wilton2010profinite}, section \S2). %, and their proof can be adapted without substantial changes to the case of compact irreducible $3$-manifold with non-empty boundary. Namely, they showed that the JSJ-splitting (\textit{i.e.} the splitting  associated with the JSJ-decomposition) of a closed $3$-manifold is $4$-acylindrical, except for the case of $Sol$-manifolds. It is clear that the statement is trivially saitisfied if the JSJ-decomposition is empty, \textit{i.e.} when the manifold is either Seifert fibered or hyperbolic, since in these cases the Bass-Serre tree consists of a single vertex.
 The main ingredients of their proof are explained in full detail in section \S2, and their approach inspired the proof of Proposition \ref{int_leaves_stab}.
%In section \S2 and partly in section \S3 of the present paper we shall explain in full details the main ingredients used in \cite{wilton2010profinite} section \S2, focusing on the relation existing between the behaviour under conjugacy of the peripheral subgroups and the existence of acylindrical splittings. 
We shall work with compact $3$-manifolds with boundary, but no claim of originality is made as Wilton-Zalesskii's  argument extends to this case without effort.
 Given an orientable, irreducible $3$-manifold $X$ we shall look to \lq\lq more elementary" splittings of $\pi_1(X)$ than the one considered in \cite{wilton2010profinite}. Namely, given the collection $\{T_i\}$ of the JSJ-tori of $X$, for each $i$ we shall study the splitting of $\pi_1(X)$ along the subgroup $\pi_1(T_i)$ as an amalgamated product or HNN-extension, rather than the splitting provided by  the whole collection. 
This gives acylindrical splittings of $\pi_1(X)$ as an amalgamated product or HNN-extension, whose constant $k$ is often smaller than the one of the JSJ-splittinng.
 The interest in finding acylindrical splittings comes for instance from Riemannian geometry and geometric group theory: we refer to \cite{cerocchi2017finiteness} and \cite{cerocchi2017entANDfin} where acylindrical splittings are exploited to prove some finiteness results for various classes of Riemannian manifolds and groups, and to \cite{cerocchi2017rigidity}  where a systolic estimate and some topogical rigidity results are proven for non-geometric $3$-manifolds.
 As we shall consider splittings of  $\pi_1(X)$ as an amalgamated product or HNN-extension, we prefer to speak of \textit{malnormal splittings}, since this does not require any reference to Bass-Serre theory.

\begin{prop}[Malnormal splittings]\label{splittings}
	Let $X$ be an orientable, irreducible $3$-manifold having a non-trivial JSJ-decomposition and which is not finitely covered by a torus bundle over the circle. Let $T$ be a JSJ-torus of $X$. We denote by $T^{\pm 1}$ the  boundary tori of $\overline{X\smallsetminus T}$  obtained by cutting $X$ along $T$, and by $X_{k_{\pm1}}$  the JSJ-components of $X$ adjacent to $T$. Let $\varphi:\pi_1(T^{-1})\f \pi_1(T^{+1})$ be the isomorphism given by the gluing.
	\begin{itemize}
		\item[(D1)] Assume that $T$ separates $X$. Let  $X_{k_{\pm1}}$ be hyperbolic JSJ-components. Then $\pi_1(X)$ splits as a $0$-step malnormal  product.
		\item[(D2)] Assume that $T$ separates $X$.  Let $X_{k_{-1}}$ be a hyperbolic JSJ-component and $X_{k_{+1}}$ be a Seifert fibered JSJ-component. Then $\pi_1(X)$ splits as a $1$-step malnormal  product.
		\item[(D3)] Assume that $T$ separates $X$. Let $X_{k_{\pm1}}$ be  Seifert fibered JSJ-components not homeomorphic to $K\widetilde{\times} I$. Then $\pi_1(X)$ splits as a $1$-step malnormal  product.
		\item[(D4)] Assume that $T$ separates $X$.  Let $X_{k_{-1}}$ be a Seifert fibered JSJ-component with hyperbolic base orbifold and $X_{k_{+1}}\simeq K\widetilde\times I$.  Then $\pi_1(X)$ splits as $3$-step malnormal product.
		\item[(ND1)] Assume that $T$ does not separate $X$. Let  $X_{k_{\pm1}}$ be two (possibly equal) hyperbolic  JSJ-components. Then $\pi_1(X)$ splits as a $1$-step malnormal HNN-extension.
		\item[(ND2)] Assume that $T$ does not separate $X$. If we are not in case {\rm (ND1)} then $\pi_1(X)$ splits as a $2$-step malnormal HNN-extension.
		\end{itemize}
\end{prop}

\begin{rmk}
	The assumption \lq\lq not finitely covered by a torus bundle" is necessary to avoid the case of $Sol$-manifolds, which do not admit this kind of splittings (see \cite{wilton2010profinite}, \S2). Actually, as a byproduct of Theorem 3.2 and of Example 5.1 in \cite{cerocchi2017rigidity}, $Sol$-manifolds do not admit any kind of acylindrical splitting.
\end{rmk}

%\begin{rmk}
%	In view of Proposition \ref{splittings} the fundamental group of a  graph manifold\footnote{A \textit{graph manifold} is a $3$-manifold with non-trivial JSJ-decomposition and such that all of its JSJ-components are Seifert fibered.} $X$ which admits a separating JSJ-torus bounding two Seifert fibered JSJ-components with hyperbolic base orbifolds provides an example of  a $1$-step malnormal amalgamated product  $A\;_{C}^{*} B$ such that the subgroup $C$ is malnormal neither in $A$ nor in $B$.
%\end{rmk}

\begin{rmk}
	Proposition \ref{splittings} is  to compare with Lemma 2.2 and Lemma 2.4 in \cite{wilton2010profinite}. Notice that it can be obtained from Wilton-Zalesskii's result by collapsing subgraphs of the JSJ-splitting as explained in \cite{kapovich2001combination}, Proposition 3.6. We shall provide a proof which does not make use of Bass-Serre theory.  It is worth to say here that analogous results have been proved for some classes of  \textit{high dimensional graph manifolds} in \cite{frigerio2015rigidity}.
    \end{rmk}

In section \S2 we introduce some background material. In particular, we shall survey topological results concerning Seifert fibered manifold, focusing on the properties of their fundamental groups.

Section \S3 is devoted to the proof of  Theorem \ref{AI_bdrtori}. We shall give a simultaneous proof of Theorem \ref{AI_bdrtori} and of Theorem 3 in \cite{delaHarpe2014onmalnormal} by looking at the Bass-Serre tree associated to a particular extension of the JSJ-splitting of the fundamental group (see Proposition \ref{int_leaves_stab}).  Key ingredients in the proof are three lemmata: Lemma \ref{Seifert} and Lemma \ref{Hyperbolic} which describes the algebraic poperties of the abelian peripheral subgroups when $X$ has an empty JSJ-decomposition, and Lemma \ref{Gluingiso} which provides a necessary and sufficient condition for a map between two boundary tori of  Seifert fibered manifolds to give rise to a graph manifold.

In section \S4 we first recall the aforementioned relation between $k$-step malnormal amalgamated product or HNN-extension and the notion of $k$-acylindrical splitting, and then we  prove Proposition \ref{splittings}.  \\

\noindent{\bf Aknowledgements.} The author wish to thank Erika Pieroni for reading the draft of this paper and for useful comments and discussions.

\section{Basic facts about $3$-manifolds and $3$-manifold groups}
 
This section is devoted to survey some classical facts of $3$-manifold geometry and topology regarding Seifert fibered manifolds, atoroidal manifolds and the JSJ-decomposition. We are mainly interested in their consequences on the algebraic structure of $3$-manifold groups. All the $3$-manifolds considered in this section and in the rest of the paper are assumed to be orientable.
 General references for the results stated in this section are the following: \cite{scott1983geometries}, \cite{fomenko1997algorithmic}, \cite{thurston1997three}, \cite{hatcher2000notes},  \cite{bonahon2002geometric},  \cite{aschenbrenner20153} (Chapters 1 and 2), \cite{martelli2016geometric}.
 
 \subsection{Seifert fibered manifolds} We shall briefly review some topological result concerning Seifert fibered $3$-manifolds. 
 
% \begin{defn}\label{defseifert}
 %	A  compact $3$-manifold $X$ is said to be \textit{Seifert fibered} if it can be decomposed into disjoint simple closed curves (the fibers of the Seifert fibration) such that each fiber has a tubular neighborhood that forms a standard fibered torus\footnote{The standard fibered torus corresponding to a pair of coprime integers $(p,q)$ with $p > 0$ is the surface bundle of the automorphism of a disk given by rotation by an angle of $\frac{2\pi q}{p}$, equipped with the natural fibering by circles. If $p = 1$, then the middle  fiber is called regular, otherwise it is called singular.}. Two Seifert fibrations are (oriented) isomorphic if there is a (orientation preserving) homeomorphism from $X$ to $X$ sending the fibers of one Seifert fibration into the fibers of the other.
% 	\end{defn}
 
%\begin{exmp}[Seifert fibrations]\label{seifertcnst} Trivial examples of Seifert fibered manifolds are  $S^1$-bundles over surfaces.
%\end{exmp}

\subsubsection{Topological construction of Seifert fibrations and their $\pi_1$.} We recall that any Seifert fibration on a $3$-manifold  can be constructed as follows. Let  $\Sigma_{g, h}$ be a surface of genus $g\in\Z$, having $h$ boundary components. Excise from $\Sigma_{g,h}$  a set of  $k\ge0$ disks $\mathcal D=\{D_1,.., D_k\}$. This gives rise to $k$ new boundary components, represented by the loops $c_1$,.., $c_k$. Let us consider $(\Sigma_{g,h}\smallsetminus\mathcal{ D})\ortimes S^1$ the (unique) orientable $S^1$-bundle over $\Sigma_{g, h}\smallsetminus\mathcal D$ and let $f$ denote the homotopy class of the fiber of the $S^1$-bundle. Glue $k$ copies of the solid torus $D^2\times S^1$ to the boundary tori of  $(\Sigma_{g,h}\smallsetminus\mathcal{ D})\ortimes S^1$ which project on the loops $c_i$ in the following way: choose a slope $p_i c_i+q_i f$, with $(p_i,q_i)$ a pair of coprime integers with $p_i>0$, and send it in the meridian loop of the $i$th solid torus $D^2\times S^1$. This procedure is known as \textit{Dehn filling} along the curve $p_i c_i+q_i f$.
 We shall denote $S(g,h; (p_1,q_1),..,(p_k,q_k))$  the Seifert fibration obtained via the previous construction.
\vspace{2mm} 
 
 %\noindent It is worth to notice that --- except for $\R P^3\#\R P^3$ --- any orientable, compact Seifert fibered $3$-manifold is prime.\\
 %Performing fiber parallel Dehn fillings is said \textit{fiber parallel} if the slope is equal to $\pm f$. It turns out that any Seifert fibered $3$-manifold (included $S^1$-bundles over surfaces) can be obtained via the previous procedure avoiding fiber parallel Dehn fillings (which give rise  to non-prime manifolds). Because of the remark, in the sequel we shall assume $p_i>0$.\\

 Let $X$ be a Seifert fibered manifold and let  $S=S(g,h; (p_1,q_1),...,(p_k,q_k))$ be a Seifert fibration on $X$. There is a presentation $\mathscr P(S)$ of $\pi_1(X)$ canonically associated to $S$  which can be obtained from the topological construction using Van-Kampen's Theorem:
\vspace{1mm}

\noindent (i) $g\ge 0$
\small
$$\mathscr P(S)=\left\langle a_1,b_1,.., b_g, c_1,..,c_k, d_1,..,d_h, f\,\left |\,\begin{array}{l}
\prod_{i=1}^{g}[a_i,b_i]\cdot\prod_{j=1}^{k} c_j\cdot\prod_{\ell=1}^{h} d_\ell=1\\ 
c_j^{p_j}f^{q_j}=1\\ 
f\leftrightarrow a_i,b_i,c_j, d_\ell
\end{array}
\right.\right\rangle$$
\normalsize

\noindent (ii) $g<0$
\small
$$\mathscr P(S)=\left\langle a_1,\dots a_{|g|}, c_1,\dots,c_k, d_1,\dots,d_h, f\;\left|\;\begin{array}{l}
\prod_{i=1}^{|g|}a_i^2\cdot\prod_{j=1}^{k} c_j\cdot\prod_{\ell=1}^{h} d_\ell=1\\
c_j^{p_j}f^{q_j}=1,\\
a_i\,f\,a_i^{-1}=f^{-1}, f\leftrightarrow c_j, d_\ell
\end{array}
\right.\right\rangle\hspace{0.8cm}$$
\normalsize

%Hence, it is easy to see that if $\Z\ni g=\mathrm{genus}(\Sigma)$ then the various $S^1$-bundle may be interpreted as Seifert fibrations given by the parameters $S(g, 0; (1,q))$,  $q\in\Z$.\\

\subsubsection{Classification of Seifert fibrations up to (oriented) isomorphisms.} It is natural to ask when two Seifert fibrations $S$ and $S'$ are oriented isomorphic, \textit{i.e.} when there exists an orientation preserving homeomorphism which sends fibers into fibers.\\ It turns out that given  $S=S(g,h; (p_1, q_1),...,(p_k,q_k))$  the following moves do not change its oriented isomorphism type (see \cite{jankins1983lectures} Theorem 1.5 or \cite{martelli2016geometric} Proposition 10.3.11):
 
\begin{itemize}
	\item[(1)] $\{(p_i,q_i),\,(p_{j}, q_{j})\}\mapsto\{(p_i, q_i-p_i), (p_{j},q_{j}+p_{j})\}$\,;
	\item[(2)] $\{(p_1,q_1),...,(p_k,q_k)\}\leftrightarrow\{(p_1,q_1),..., (p_k,q_k), (1,0)\}$\,;
	\item[(3)] $(p_i,q_i)\mapsto (p_i, q_i\pm p_i)$, if $h>0$\,;
	\item[(4)] permutations of the indices of the collection of pairs $\{(p_i,q_i)\}$\,.
\end{itemize}

Moreover, two Seifert fibrations are oriented isomorphic if and only if they are related via a finite sequence of these moves (see \cite{martelli2016geometric}, Proposition 10.3.11). 

\begin{rmk}
    The Seifert fibrations 
    \small
  $$S(g,h; (p_1,q_1),..,(p_k,q_k)),\mbox{ }\,S(g,h; (p_1,-q_1),.., (p_k,-q_k))$$
  \normalsize
     are isomorphic via the homeomorphism which sends fibers into fibers with the opposite orientation, but they are not in general oriented isomorphic.
\end{rmk}

These moves allow to classify Seifert fibered manifolds up to (oriented) isomorphism, but before to give the precise statement we need to introduce a useful invariant: the \textit{Euler number} of a Seifert fibration.
 It is a well known fact that the oriented $S^1$-bundles over a closed surface $\Sigma_g$ can be realized as Seifert fibrations $\{S(g, 0; (1,q))\}_{q\in\Z}$ (see for instance \cite{montesinos1987classical}, Chapter 1). A classical result of algebraic topology says that oriented $S^1$-bundles over closed surfaces are classified by their Euler number, \textit{i.e.} the integral over the base surface of a suitable characteristic class, the Euler class. The Euler number can be seen as an obstruction for the $S^1$-bundle to admit a global section. Given the $S^1$-bundle $S(g,0; (1,q))$, we know that, by construction,  it has a section $s$ in $\Sigma_g\smallsetminus\{x\}$ where $x$ is the projection of the central fiber of the fibered solid torus. It turns out that the Euler number coincides (\cite{bott1982differentialforms}, Theorem 11.16) with the local degree of $s$ at $x$, \textit{i.e} with $q$.
Given a general Seifert fibration $S=S(g,h; (p_1,q_1),..,(p_k,q_k))$, we define its Euler number as  $e(S)=\sum_1^k\frac{q_i}{p_i}$, where the previous sum is defined only mod $\Z$ when $h\neq 0$. Notice that, by definition, the Euler number (possibly mod $\Z$) is invariant under the moves (1)-(4). This definition of the Euler number coincides for $S^1$-bundles over closed surfaces with the classical one.
The following proposition holds (see \cite{jankins1983lectures} Theorem 1.5 or  \cite{martelli2016geometric}, Corollary 10.3.13):

\begin{prop}[Classification of Seifert fibration up to isomorphism]\label{classuptoiso}
	Consider $S=S(g,h;(p_1,q_1),..,(p_k,q_k))$ and $S'=S(g', h'; (p_1',q_1'),..,(p_{k'}',q_{k'}'))$, two Seifert fibrations with $p_i, p_i'\ge 2$. They are (orientation preservingly) isomorphic if and only if $g=g'$, $h=h'$, $k=k'$, $e(S)=e(S')$, and, up to reordering the indices, $p_i=p_i'$, $q_i=q_i' \;\mathrm{mod}\, p_i$.
	\end{prop}

\begin{rmk} Let  $S'=S(g,h;(p_1',q_1'),..,(p_{k'}',q_{k'}'))$ with $h>0$ be a Seifert fibration. A consequence of Proposition \ref{classuptoiso} is that either $S'$ is isomorphic to  $S=S(g,h; (p_1,q_1),...,(p_k,q_k))$ with $0<q_i<p_i$ or $S'$ is isomorphic to $S=S(g,h;\,)$, and the two possibilities are mutually exclusive. 
	These Seifert fibrations are canonical in their isomorphism class. In fact, they are  the unique enjoying these properties.
	\end{rmk}

\subsubsection{Classification of Seifert manifolds up to diffeomorphism.} When does a Seifert manifold $X$ admit non-isomorphic Seifert fibrations? The next result (\cite{jaco1980lectures} Theorems VI.17 and VI.18) answer the question: 

\begin{prop}[Classification of Seifert fibered manifolds up to diffeomorphism]\label{classupdiffeo}
	Every Seifert fibered manifold with non-empty boundary admits a unique fibration up to isomorphism, except in the following cases:
	\begin{itemize}
		\item $D^2\times S^1$, which admits the fibrations $S=S(0,1; (p,q))$ for any pair of coprime integers $(p,q)$;
		\item $K\widetilde{\times}I$, which admits two non-isomorphic Seifert fibrations:\\ $S_1=S(-1,1;\,)$, $S_2=S(1,1; (2,1), (2,1))$.
	\end{itemize}
	Every closed Seifert fibered manifold $X$ which is not covered by $S^3$ or $S^2\times S^1$ admits a single Seifert fibration up to isomorphism, except for $K\widetilde\times S^1$ which admits the non-isomorphic fibrations   $S_3=S(0,0; (2,1), (2,1), (2,-1), (2,-1))$ and  $S_4=S(-2,0;\,)$.
\end{prop}

\subsubsection{Seifert fibrations as $S^1$-bundles over $2$-dimensional orbifolds.} There is a second way to look at (oriented) Seifert fibered manifolds: %as the Seifert fibrations $S(g,h; (1,q))$ admit an interpretation as $S^1$-bundles over compact surfaces, similarly  general Seifert fibrations
they can be seen as $S^1$-bundles over $2$-dimensional orbifolds having only conical singular points. We refer to the classical references \cite{scott1983geometries} and \cite{thurston1997three} for an account on $2$-dimensional orbifolds. It is worth to stress that this point of view is particularly helpful in order to show the existence of complete, locally homogeneous metrics on Seifert fibered manifolds (\cite{scott1983geometries}, \cite{ohshika1987teichmuller}, \cite{bonahon2002geometric}).\\
Given $S=S(g, h; (p_1, q_1),..., (p_k,q_k))$, we  associate to $S$ its base orbifold $\mathcal O_S$. The orbifold $\mathcal O_S$ has as underlying topological space the surface of genus $g$ and $h$ boundary components, and $k$ conical singular points of order $p_1$,.., $p_k$. There is a natural extension of the usual Euler characteristic to $2$-dimensional orbifolds, known as the \textit{orbifold characteristic} (see \cite{scott1983geometries} or \cite{bonahon2002geometric}), denoted $\chi_{orb}$. The orbifold characteristic allows in first instance to distinguish between the orbifolds admitting a manifold cover (\textit{good orbifolds}) and those who do not admit such a cover (\textit{bad orbifolds}). A complete list of bad $2$-dimensional orbifolds can be found in \cite{scott1983geometries}. In particular, the bad $2$-dimensional orbifolds which may appear as base orbifolds of Seifert fibered manifolds are either of type $S^2(p)$ (underlying surface $S^2$ and a single conical singular point of order $p$) or of type $S^2(p,q)$ (underlying surface $S^2$ and two conical singular point of order $p$ and $q$ with $\mathrm{gcd}(p,q)=1$). The Seifert fibered manifolds which fibers over bad  $2$-dimensional orbifolds belong to a specific and well understood class of quotients of the $3$-sphere: the class of \textit{lens spaces} (see \cite{jankins1983lectures}, section I.4) which in particular have fundamental groups isomorphic to finite cyclic groups.
Secondly,  good orbifolds carry geometric structures and the orbifold characteristic detects which kind of geometry a good orbifold carries: namely, good orbifolds are divided into \textit{elliptic}, \textit{Euclidean} and \textit{hyperbolic orbifolds} depending whether $\chi_{orb}>0$, $\chi_{orb}=0$ or $\chi_{orb}<0$ (see \cite{thurston1997three}, Theorem 13.3.6).\\ %It turns out that good compact $2$-orbifolds are finitely covered by surfaces: this can be seen in \cite{scott1983geometries} for closed orbifolds, and follows from \cite{edmonds1982torsionfree} (see also \cite{burns1983theindices}, Theorem 1) and an easy case by case analysis for the remaining orbifolds, in case the boundary is non-empty.\\
%Notice that, since we are interested in orientable $3$-manifolds, the base orbifolds which appear as bases for orientable Seifert fibered manifolds do not have any other kind of singular point. 

Our focus is on Seifert fibered manifolds with non-empty boundary, which clearly fibers over $2$-orbifolds with boundary. The set of $2$-orbifolds with boundary consists of the collection of disks having a unique singular point, three Euclidean orbifolds and hyperbolic orbifolds. The Euclidean orbifolds are: $D^2(2,2)$, the disk with two singular points of order $2$, $Mb$, the M\"obius strip, and $C$, the cylinder. Notice that there are only two Seifert fibered manifolds fibering over these orbifolds: $K\widetilde{\times} I$, fibering over $D^2(2,2)$ or over $Mb$, and $T^2\times I$, fibering over C. Finally, there is a unique Seifert fibered manifold fibering over the collection of disks with a unique singular point: the solid torus $D^2\times S^1$. Hence, except for $K\widetilde\times I$, $T^2\times I$, $D^2\times S^1$, all Seifert fibered $3$-manifolds with boundary are fibered over hyperbolic base orbifolds.

\subsubsection{The short exact sequence.} The description of Seifert fibered manifolds as $S^1$-bundles over $2$-dimensional orbifolds provides a short exact sequence (see \cite{scott1983geometries}):
$$1\f\langle f\rangle\f\pi_1(S)\f{\pi_1}^{orb}(\mathcal O_S)\f 1$$
where $\pi_1^{orb}(\mathcal O_S)$ is the \textit{orbifold fundamental group} (for the definition see for instance \cite{boileau2003threedim}, Chapter 2) and the surjective morphism corresponds to the quotient by the normal subgroup of the fiber of $S$. We thus obtain the following presentations for the orbifold fundamental groups:\\
%Recall now the explicit expression  for the fundamental group of $S$. Using the fact that the surjective morphism $\pi:\pi_1(S)\f\pi_1^{orb}(\mathcal O_S)$ in the exact sequence is the quotient of $\pi_1(S)$ by the normal subgroup of the regular fiber, it is easily seen that we have this kind of presentations for the orbifold fundamental group:
\small
$$\pi_1^{orb}(\mathcal O_S)=\left\langle \bar a_1,\bar b_1,.., \bar b_g, \bar c_1,..,\bar c_k, \bar d_1,..,\bar d_h\,\left |\,\begin{array}{l}
\prod_{i=1}^{g}[\bar a_i,\bar b_i]\cdot\prod_{j=1}^{k} \bar c_j\cdot\prod_{\ell=1}^{h} \bar d_\ell=1\\ 
\bar c_i^{p_i}=1
\end{array}
\right.\right\rangle$$
$$\pi_1^{orb}(\mathcal O_S)=\left\langle \bar a_1,...\,\bar a_{|g|},  \bar c_1,...,\bar c_k, \bar d_1,...,\bar d_h\,\left |\,\begin{array}{l}
\prod_{i=1}^{|g|}\bar a_i^2\cdot\prod_{j=1}^{k} \bar c_j\cdot\prod_{\ell=1}^{h} \bar d_\ell=1\\ 
\bar c_i^{p_i}=1
\end{array}
\right.\right\rangle$$
\normalsize
depending on the sign of $g$, \textit{i.e.} depending on the orientability of the topological space underlying the base orbifold. In the previous presentations is understood that $p(a_i)=\bar a_i$, $p(b_i)=\bar b_i$, $p(c_j)=\bar c_i, p(d_\ell)=\bar d_\ell$, where $p:\pi_1(S)\f \pi_1^{orb}(\mathcal O_S)$ is the surjective morphism of the exact sequence.

\begin{lem}\label{peripheral_orb}
	If $h>0$ and $\mathcal O_S$ is a hyperbolic orbifold, then the collection of subgroups $\{\langle\bar d_i\rangle\}_{i=1}^h$ is a collection of malnormal, conjugately separated subgroups in $\pi_1^{orb}(\mathcal O_S)$
\end{lem}

\begin{proof}
	Since $\mathcal O_S$ is a compact $2$-orbifold with non empty boundary  we can express $d_h$ in terms of the other generators. We see that $\pi_1^{orb}(\mathcal O_S)\cong \mathbb F_{2g+h-1}*\Z_{p_1}*\cdots \Z_{p_k}$ freely generated by $\{\bar a_i, \bar b_i\}_{i=1}^g\cup\{\bar d_\ell\}_{\ell=1}^{h-1}\cup\{\bar c_j\}_{j=1}^k$ if the genus $g$ of $|\mathcal O|$ is positive  and $\pi_1^{orb}(\mathcal O_S)\cong\mathbb F_{g+h-1}*\Z_{p_1}*\cdots *\Z_{p_k}$ freely generated by $\{\bar a_i\}_{i=1}^{|g|}\cup\{\bar d_\ell\}_{\ell=1}^{h-1}\cup\{\bar c_j\}_{j=1}^k$ if the genus $g$ of $|\mathcal O_S|$ is negative. Moreover, as the orbifold $\mathcal O_S$ has negative orbifold Euler characteristic we see that the previous free products are always different from $\Z_2*\Z_2$. Observe that $\{\langle \bar d_i\rangle\}_{i=1}^{h-1}$ is a collection of malnormal, conjugately separated infinite cyclic subgroups. In fact, they represent by definition $h-1$ infinite cyclic free factors of the free product which determines $\pi_1^{orb}(\mathcal O_S)$, because $g\bar d_ig^{-1}$ has syllable length greater than $1$, unless $g=	\bar d_i^{k}$ for some $k\in\Z$. On the other hand, writing down the expression for $\bar d_h$ in terms of the other generators we see that it is a primitive element of infinite order in a free product of cyclic groups, which can be represented as a cyclically reduced word of length strictly greater than one. It follows that $\bar d_h$ is not conjugate to any other $\bar d_i$ (because their cyclically reduced length is equal to $1$). On the other hand, if $g\bar d_hg^{-1}=\bar d_h^{k}$ then we would have $g^2\bar d_hg^{-2}=\bar d_h$. Since the centralizer of a primitive element of cyclically reduced length greater than $1$ in a free product is infinite cyclic, generated by the primitive element we see that $g^2=\bar d_{h}^k$. But in a free product of cyclic groups any element of infinite order possess a unique root,  and since $\bar d_h$ is primitive we conclude that $k=2m$ and $g=\bar d_h^m$, proving that $\langle \bar d_h\rangle$ is malnormal in $\pi_1^{orb}(\mathcal O_S)$.
	%If $h>0$ we are allowed to express $\bar d_h$ in terms of the other generators. Thus $\pi_1^{orb}(\mathcal O_S)$ is isomorphic to $\mathbb F_{2g+h-1}*\Z_{p_1}*\cdots*\Z_{p_k}$ (freely generated by the elements $\{\bar a_i, \bar b_i\}_{1}^g$, $\{\bar c_i\}_1^k$, $\{\bar d_j\}_1^{h-1}$) if $g>0$ and to $\mathbb F_{|g|+h-1}*\Z_{p_1}*\cdots*\Z_{p_k}$ (freely generated by the elements $\{\bar a_i\}_{1}^{|g|}$, $\{\bar c_i\}_1^k$, $\{\bar d_j\}_1^{h-1}$)  if $g<0$. Since in our case $\mathcal O_S$ is hyperbolic, then the previous free products are non-trivial, and not isomorphic to $\Z_2*\Z_2$. Using Kurosh's subgroup theorem it is easy to show that the peripheral subgroups $\{\langle\bar d_i\rangle\}_1^{h}$ form a collection of malnormal subgroups. To prove that they are conjugately separated subgroups of $\pi_1^{orb}(\mathcal O_S)$ we use the fact that the centralizer of any element of infinite order in a free product of cyclic groups is infinite cyclic. 
\end{proof}

\begin{rmk}[The regular fiber] \label{regular_fiber} Let $X$ be a Seifert fibered $3$-manifold with non-empty boundary, not homeomorphic to $T^2\times I$ and $D^2\times S^1$.  We recall (subsection \S2.1.1) that any Seifert fibration of $X$ determines a presentation of $\pi_1(X)$ and a maximal, infinite cyclic, normal subgroup: the subgroup corresponding to the homotopy class of the regular fibers of the Seifert fibration.  From Proposition \ref{classupdiffeo} we know that if $X$ has non-empty boundary and it is not homeomorphic to $T^2\times I$, $D^2\times I$ or $K\widetilde{\times} I$, all Seifert fibrations are isomorphic. On the other hand, by the explicit description of the fundamental group of a Seifert fibered manifold and arguments similar to the ones used in Lemma \ref{peripheral_orb} it is not difficult to check that $\pi_1(X)$ possesses a unique maximal, infinite cyclic normal subgroup. Hence, despite the fact that distinct Seifert fibrations of $X$ in the same isomorphism class give rise to distinct presentations of $\pi_1(X)$, they all determine the same maximal, infinite cyclic, normal subgroup. In view of this fact we shall call the generator --- univoquely defined up to take inverses --- of this unique maximal, infinite cyclic, normal subgroup \textit{the regular fiber of $\pi_1(X)$}, since this element  does no longer depend on the particular fibration.\\
Similarly, if we look at $K\widetilde{\times}I$ there are only two isomorphism classes of Seifert fibrations, which are identified by the only two distinct maximal, infinite cyclic, normal subgroups in $\pi_1(K\widetilde{\times}I)=\langle a, f\,|\, afa^{-1}=f^{-1}\rangle$, namely the subgroups $\langle a^2\rangle$ and $\langle f\rangle$. We shall call $a^2$ and  $f$ \textit{the regular fibers of $K\widetilde{\times} I$}.
\end{rmk}

The next lemma states that the peripheral subgroups $\langle d_i, f\rangle=p^{-1}(\langle \bar d_i\rangle)$ of $S$ partially inherit the properties of the peripheral subgroups of $\mathcal O_S$:

\begin{lem}\label{Seifert}
	Let $X$ be a Seifert fibered manifold with non-empty boundary which fibers over a hyperbolic base orbifold and let $f\in\pi_1(X)$ be the regular fiber.
	\begin{itemize} 
		\item[(i)] If $T^{\bdr}\subset\bdr X$ is a boundary torus and  $g\in\pi_1(X)\smallsetminus\pi_1(T^{\bdr})$, then
		 $$g\,\pi_1(T^{\bdr})\,g^{-1}\cap\pi_1(T^{\bdr})=\langle f\rangle$$
		
		\item[(ii)] If $T_{1}^{\bdr},\,T_{2}^{\bdr}\subset\bdr X$ are two  boundary tori and   $g\in\pi_1(X)$, then $$g\pi_1(T_{1}^{\bdr})g^{-1}\cap\pi_1(T_{2}^{\bdr})=\langle f\rangle$$
	\end{itemize}
\end{lem}

The previous lemma is implicit in  Wilton-Zalesskii's proof of the existence of acylindrical splittings for irreducible $3$-manifold groups (\cite{wilton2010profinite}, section \S2) and will play an important role  in the proof of Theorem \ref{AI_bdrtori}. It can be proved via combinatorial group theory arguments using Lemma \ref{peripheral_orb} and the short exact sequence.

\subsubsection{Further remarks about $\pi_1(K\widetilde{\times}I)$.} We conclude this section on Seifert fibered manifolds  resuming in the next lemma few straightforward facts about the peripheral subgroups and the regular fibers of $\pi_1(K\widetilde{\times}I)$.

\begin{lem}\label{kappatwistato}
	Let $\pi_1(K\widetilde{\times} I)=\langle a, f\, |\, afa^{-1}=f^{-1}\rangle$ be the presentation associated to the Seifert fibration $S(-1,1;)$ of $K\widetilde{\times}I$. Then:
	\begin{itemize}
		\item[(i)] $\pi_1(\bdr(K\widetilde\times I))=\langle a^2, f\rangle$ is a normal subgroup of $\pi_1(K\widetilde{\times} I)$;
		\item[(ii)] the two regular fibers of $\pi_1(K\widetilde{\times}I)$ are $a^2$ and $f$;
		\item[(iii)] if $g=a^{2p}f^q$ (with $\mathrm{gcd}(p,q)=1$) is a primitive element in $\pi_1(K\widetilde{\times} I)$ which is not equal to $a^{\pm 2}$ or $f^{\pm 1}$, then $a^{\pm 1}\langle a^{2p}f^q\rangle a^{\mp 1}\cap\langle a^{2p} f^q\rangle=\{1\}$.
	\end{itemize}
	\end{lem}
 
 \subsection{ The JSJ-decomposition Theorem}\label{JSJDEC}
One of the cornerstones of $3$-manifolds theory  is the existence of a decomposition for compact, irreducible $3$-manifolds along embedded, non boundary parallel, incompressible tori (\cite{jaco1979seifert}, \cite{johannson1979homotopy}):

\begin{JSJdecthm}
 Let $X$ be a compact, irreducible 3-manifold.
In the interior of $X$, there exists a family $\mathcal C=\{T_1, . . . , T_r\}$ of disjoint tori that are incompressible and not boundary parallel, with the following properties:
\begin{itemize}
\item[(i)] each connected component of $X\smallsetminus\mathcal C$ is either a Seifert manifold or is atoroidal; 
\item[(ii)] the family $\mathcal C$ is minimal among those satisfying {\rm (i)}.
\end{itemize}
Moreover, such a family $\mathcal C$ is unique up to ambient isotopy.
\end{JSJdecthm}

%\noindent  Notice that Seifert fibered $3$-manifolds can be atoroidal. The list of atoroidal Seifert fibered $3$-manifolds can be found in Jaco-Shalen (\cite{jaco1979seifert}, IV.2.5, IV.2.6). %and  we restrict our attention to the Seifert fibered manifolds with non-empty boundary, we find these manifolds: $K\widetilde{\times}I$, $T^2\times I$ and the set of Seifert fibrations $\{S(1,1; (p_1,q_1), (p_2,q_2)\}_{\max\{p_1,p_2\}\ge 3}$. Notice that the base orbifolds of the Seifert lfibrations $\{S(1,1; (p_1,q_1), (p_2,q_2)\}_{\max\{p_1,p_2\}\ge 3}$ are hyperbolic.\\

Following Thurston \cite{thursto1982kleinian}, we say that an irreducible $3$-manifold $X$ is \textit{homotopically atoroidal} if every $\pi_1$-injective map from the torus to the irreducible manifold is homotopic to a map into the boundary. Being homotopically atoroidal is a stronger property than being atoroidal, where we consider embeddings instead of $\pi_1$-injective maps. Nevertheless, the two notions agree except for some Seifert fibered manifolds. In fact, the atoroidal, Seifert fibered manifolds  are the $3$-manifolds which admit one of the Seifert fibrations from the list IV.2.5 in \cite{jaco1979seifert} (see \cite{jaco1979seifert} Lemma IV.2.6). Among the manifold in the list only a few of them are homotopically atoroidal. To our purposes it is sufficient to say that the list of compact, homotopically atoroidal, Seifert fibered manifold with non-empty boundary is reduced to the following list\footnote{In order to check which closed Seifert fibered manifold among the atoroidal are homotopically atoroidal a way to proceed is to compute the fundamental groups and consider the ones which do not admit a subgroup isomorphic to $\Z^2$.}: $K\widetilde{\times} I$, $T^2\times I$, $D^2\times S^1$.\\
In 1982 (\cite{thursto1982kleinian}) Thurston proposed the Geometrization conjecture and announced a series of papers proving the Hyperbolization Theorem for Haken $3$-manifolds\footnote{We recall that an Haken $3$-manifold is an irreducible $3$-manifolds which contains a properly embedded, $2$-sided, incompressible surface.} (\cite{thurston1986hyperbolic}, \cite{thurston1998hyperbolicII}, \cite{thurston1998hyperbolicIII}  ---the last two are unpublished---):
\begin{hyperbolization}[Thurston]
	The interior of a  compact, irreducible, Haken $3$-manifold admits a complete hyperbolic metric if and only if the manifold is homotopically atoroidal and not homeomorphic to $K\widetilde{\times} I$.
\end{hyperbolization}
 A complete proof of Thurston's Hyperbolization Theorem can be found in \cite{otal1996theorem}, \cite{otal1998thurston} and \cite{kapovich2001hyperbolic}.  Since compact, irreducible $3$-manifolds with non-empty boundary and not homeomorphic to $D^2\times S^1$ are Haken manifolds, Thurston's result implies that such a $3$-manifold admits a complete hyperbolic metric if and only if is homotopically atoroidal and not homeomorphic to $K\widetilde{\times}I$. Moreover, if the $3$-manifold has only toroidal boundary and in addition is not homeomorphic to $T^2\times I$, the complete hyperbolic metric has finite volume. Thanks to the work of Perelman (\cite{perelman2002entropy}, \cite{perelman2003ricci}, \cite{perelman2003finiteext}) we are allowed to replace the assumption \lq\lq\textit{Haken}" with the assumption \lq\lq\textit{with infinite fundamental group}".
Coherently with the picture provided by the Hyperbolization Theorem it is customary to call \textit{manifolds of hyperbolic type} the homotopically atoroidal, irreducible $3$-manifolds (possibly with non empty boundary) not homeomorphic to $D^2\times S^1$, $T^2\times I$ or $K\widetilde{\times}I$.\\% \textit{i.e.} the non-Seifert fibered, homotopically atoroidal, compact, irreducible $3$-manifolds with non-empty boundary.\\

\noindent Assume that $X\neq T^2\times I, D^2\times S^1$  and $\bdr X\neq\varnothing$. Then  $T^2\times I, D^2\times S^1$ cannot appear as a JSJ-component: the first by minimality of the collection $\mathcal C$; the second because a JSJ-torus bounding $D^2\times S^1$ would not be incompressible. Hence if $X_k$  is a JSJ-component of $X$ one of the following mutually exclusive conditions holds:
\begin{itemize}
	\item $X_k$ is of hyperbolic type;
	\item $X_k$  is Seifert fibered with hyperbolic base orbifold;
	\item $X_k$ is homeomorphic to $K\widetilde{\times}I$.
\end{itemize}

%Observe that we do not assume the irreducible $3$-manifold $X$ to have incompressible boundary. In fact, except for $X=D^2\times S^1$, which is the unique Seifert fibered manifold having a compressible torus as boundary, the boundary tori of irreducible $3$-manifolds are always incompressible.% Moreover, boundary components (possibly compressible) of genus strictly greater than $1$ are necessarily contained into hyperbolic JSJ-components.\\

The next lemma is well known (see for example \cite{aschenbrenner20153}, Lemma 1.5.3) and it gives a necessary and sufficient condition for the gluing of two Seifert fibered $3$-manifolds not to be Seifert fibered.

 \begin{lem}\label{Gluingiso}
	Let $X, Y\neq D^2\times S^1$ be two  Seifert fibered manifolds (possibly $X=Y$) with non-empty boundary. Let $g: T_X^{\bdr}\f T_Y^{\bdr}$ be a diffeomorphism between two boundary tori of $X$. Let  $Z$ be the resulting manifold.  The following are equivalent:
	\begin{itemize}
		\item[(i)] the manifold $Z$ is not Seifert fibered;
		\item[(ii)]  for any choice  of regular fibers $f_X$ in $\pi_1(X)$ and $f_Y$  in $\pi_1(Y)$ the following condition holds:  $g_*(\langle f_X\rangle)\cap\langle f_{Y}\rangle=\{1\}$.
	\end{itemize}
\end{lem}

%\begin{rmk}
%	For easier reference in the sequel we shall call \textit{JSJ-gluing}  a diffeomorphism between two  boundary tori of Seifert fibered manifolds satisfying condition (ii) of Lemma \ref{Gluingiso}.
%\end{rmk}

Finally we  record a result concerning the abelian peripheral subgroups of manifolds of hyperbolic type. In analogy with the case of  hyperbolic $2$-orbifolds with boundary (Lemma \ref{peripheral_orb}), the collection of the abelian peripheral subgroups of an irreducible $3$-manifold of hyperbolic type are malnormal and conjugately separated:

\begin{lem}\label{Hyperbolic}
		Let $X$ be an irreducible $3$-manifold of hyperbolic type with non-empty boundary and let $T^{\bdr}, T_1^{\bdr}, T_2^{\bdr}$ be toroidal boundary components in $\bdr X$.
		\begin{itemize}
		\item[(i)] if $g\in\pi_1(X)\smallsetminus\pi_1(T^{\bdr})$, then $g\pi_1(T^{\bdr})g^{-1}\cap\pi_1(T^{\bdr})=\{1\}$;
	    	\item[(ii)] if $g\in\pi_1(X)$, then $g\pi_1(T_{1}^{\bdr})g^{-1}\cap\pi_1(T_{2}^{\bdr})=\{1\}$.
	    \end{itemize}
\end{lem}

\noindent A reference for the previous result is  \cite{delaharpe2014malnormalsubgroups}, \S3, Example 6.

%Notice that in this case (ii) follows from (i): in fact, if $g\pi_1(T_1^{\bdr})g^{-1}\cap\pi_1(T_2^{\bdr})\neq 1$, it follows from the fact that $\pi_1(T_1^{\bdr})$ is abelian that each of the elements of $g\pi_1(T_i^{\bdr})g^{-1}$ stabilizes the intersection of the two subgroups. By malnormality of $\pi_1(T_2^{\bdr})$ it follows that $g\pi_1(T_1^{\bdr})g^{-1}=\pi_1(T_2^{\bdr})$, which is not possible since the two boundary tori are distinct and $X\neq T^2\times I$.\\

%For what concerns assertion (i) of Lemma \ref{Hyperbolic}, notice that the peripheral subgroups $\{\pi_1(T_{1}^{\bdr})\}$ associated with the boundary tori of an hyperbolic manifold $X$ correspond to the maximal parabolic subgroups $\{P_i\}$ (each fixing a suitable point $\xi_i$ in the sphere at infinity) of a representation $\Gamma\subset PSL(2,\C)$ of $\pi_1(X)$ into $PSL(2,\C)$. The subgroups $\{P_i\}$ are known to be malnormal in $\Gamma$ (and hence conjugately separated), a property which is invariant by isomorphism.\\

\section{Abelian peripheral subgroups of irreducible $3$-manifolds}

Throughout this section $X$ will be an irreducible $3$-manifold with non-empty boundary, not homeomorphic to $K\widetilde{\times}I$, $T^2\times I$, $D^2\times S^1$. 
 Let $\{X_k\}_{k=1}^{N}$ be the JSJ-components of $X$, let  $\{T_{(k,i)}^{\bdr}\}_{i=1}^{m_k}$  be the set of the boundary tori of $X$ which belong to the JSJ-component $X_k$ (possibly $m_k=0$, if the previous collection is empty) and $\{T_{(k_1,k_2)}^{j}\}_{j=1}^{n_{(k_1,k_2)}}$  ($k_1\le k_2$, possibly $n_{(k_1,k_2)}=0$) be the collection of the JSJ-tori which bound both the JSJ-components $X_{k_1},\, X_{k_2}$.\\
 We shall also assume to have fixed:
 \begin{itemize}
  \item a basepoint $x_0$ in $X$, basepoints $x_k$ in the interior of each JSJ-component $X_{k}$, and paths from $x_0$ to the basepoints $x_k$;
\item a basepoint for each boundary torus of the JSJ-component $X_k$ and paths connecting them to the chosen basepoint $x_k\in X_k$.
\end{itemize}
With this proviso the fundamental groups of the tori $T_{(k,i)}^{\bdr}$, $T_{(\ell, k)}^{j}$, $T_{(k,\ell)}^r$ are naturally identified with suitable subgroups of $\pi_1(X_k)$, and the fundamental groups of the JSJ-components are naturally identified with subgroups of $\pi_1(X)$.

\subsection{The peripheral extension of the JSJ-splitting} We shall first recall the definition of the JSJ-splitting of the fundamental group of the irreducible $3$-manifold $X$. For the general theory of graph or groups we refer to \cite{bass1976remarks}, \cite{serre1977arbres}, \cite{serre1980trees}.

\begin{defn}[JSJ-splitting]
We define the JSJ-splitting of $\pi_1(X)$ as the graph of groups $(\mathscr G_X,\Gamma_X)$ of $\pi_1(X)$ (with the assignement of  an orientation for the edges) defined in the following way:
\begin{itemize}
	\item $\Gamma_X$ is the graph defined by the vertices 
	\small $\mathsf V(\Gamma_X)=\{\mathsf v_k\}_{k=1}^{N},$\normalsize\, and the edges \small$\mathsf E(\Gamma_X)=\left\{\mathsf e_{(k_1, k_2)}^j, \overline{\mathsf e}_{(k_1, k_2)}^j \;|\; k_1\le k_2,\;k_1=1,..,N\mbox{ and }j=1,..,n_{(k_1,k_2)}\right\}$
	\normalsize
	where $\overline{\mathsf e_{(k_1, k_2)}^j}=\overline{\mathsf e}_{(k_1,k_2)}^j$ and where $s(\mathsf e_{(k_1,k_2)}^j)=\mathsf v_{k_1}$, $t(\mathsf e_{(k_1, k_2)}^j)=\mathsf v_{k_2}$ (given an oriented edge $\mathsf e$ we define by $s(\mathsf e)$ and $t(\mathsf e)$ respectively the source and the target  of the edge $\mathsf e$). We choose the following orientation of $(\Gamma_X,\mathscr G_X)$: $\mathsf E_{+}(\Gamma_X)=\{\mathsf e_{(k_1,k_2)}^j\;|\;k_1\le k_2\mbox{ and }j=1,..,n_{(k_1,k_2)}\}$. \\
	\item The collection of subgroups $\mathscr G_X$ is given by
	\small
	 $$G_{\mathsf v_k}=\pi_1(X_k),\quad G_{\mathsf e_{(k_1,k_2)}^j}=G_{\overline{\mathsf e}_{(k_1,k_2)}^j}=\pi_1(T_{(k_1,k_2)}^j)\cong\Z^2$$
	 \normalsize
	 and the monomorphisms 
	 \small
	 $$\varphi_{(k_1,k_2; j)}:G_{\mathsf e_{(k_1,k_2)}^j}\f G_{\mathsf v_{k_1}},\;\;\; \overline \varphi_{(k_1,k_2; j)}: G_{\mathsf e_{(k_1,k_2)}^j}\f G_{\mathsf v_{k_2}}$$
	 \normalsize
	 are such that the composition $\psi_{(k_1,k_2; j)}=\overline{\varphi}_{(k_1,k_2; j)}\circ\varphi_{(k_1,k_2; j)}^{-1}$ coincides with the isomorphism induced by the gluing map between the two boundary tori of the JSJ-components $X_{k_1}$, $X_{k_2}$ giving rise to the JSJ-torus $T_{(k_1,k_2)}^j$.\\
\end{itemize}
\end{defn}
\noindent We shall now define the announced extension of the graph of groups $(\mathscr G_X, \Gamma_X)$:

 \begin{defn}[Peripheral extension of the JSJ-splitting]
We shall denote by  $(\widehat{\mathscr G}_X,\widehat{\Gamma}_X)$  the \textit{peripheral extension of the JSJ-splitting of $X$}, \textit{i.e.} the graph of groups  defined as follows:
\begin{itemize}
	\item $\widehat{\Gamma}_X$ is the graph given  by the vertices
	\small
	$$\mathsf V(\widehat{\Gamma}_X)=\mathsf V(\Gamma_X)\cup\{\mathsf v_{(k,i)}^{\bdr}\;|\; k=1,...\ell\mbox{ and }i=1,..,m_k\}$$
	\normalsize and the edges
	\small
	$\mathsf E(\widehat{\Gamma}_X)=\mathsf E(\Gamma_X)\cup\{\mathsf e_{(k,i)}^{\bdr}, \overline{\mathsf e}_{(k, i)}^{\bdr}\;|\;k=1,..,\ell\mbox{ and }i=1,..,m_k\}$
	\normalsize
	where $\overline{\mathsf e_{(k,i)}^{\bdr}}=\overline{\mathsf e}_{(k,i)}^{\bdr}$ and  $s(\mathsf e_{(k,i)}^{\bdr})=\mathsf v_{(k,i)}^{\bdr}$, $t(\mathsf e_{(k,i)}^{\bdr})=\mathsf v_k$. The orientation of $(\widehat{\mathscr G}_X,\widehat{\Gamma}_X)$ is:\,
	\small
	$\mathsf E_+(\widehat{\Gamma}_X)=\mathsf E_+(\Gamma_X)\cup\left\{\mathsf e_{(k,i)}^{\bdr}\;|\;k=1,...,\ell\mbox{ and }i=1,...,m_k\right\}$.\\
	\normalsize
	\item The collection of subgroups $\widehat{\mathscr G}_X$ is the union:
	\small
	$$\widehat{\mathscr G}_X=\mathscr G_X\cup\{G_{\mathsf v_{(k,i)}^{\bdr}}, G_{\mathsf e_{(k,i)}^{\bdr}}=G_{\overline{\mathsf e}_{(k,i)}^{\bdr}}\;|\; k=1,...,\ell\mbox{ and }i=1,..,m_k\}$$
	\normalsize
	where $G_{\mathsf v_{(k,i)}^{\bdr}}\cong\Z^2$,\, $ G_{\mathsf e_{(k,i)}^{\bdr}}= G_{\overline{\mathsf e}_{(k,i)}^{\bdr}}\cong \Z^2$
	and the monomorphisms 
	\small
	$$\varphi_{(k,i)}: G_{\mathsf e_{(k,i)}^{\bdr}}\f G_{\mathsf v_{(k,i)}^{\bdr}},\;\;\overline{\varphi}_{(k,i)}:G_{\mathsf e_{(k,i)}^{\bdr}}\f G_{\mathsf v_k}$$
	\normalsize
	are such that $\psi_{(k,i)}=\overline{\varphi}_{(k,i)}\circ\varphi_{(k,i)}^{-1}$ is the isomorphism induced by the  inclusion $T_{(k,i)}^{\bdr}\hookrightarrow X_k$.
	\end{itemize}
 \end{defn}
Given $\widehat{\mathsf T}_X\subseteq\widehat \Gamma_X$ a maximal subtree of $\widehat{\Gamma}_X$ we recall that is well defined the fundamental group $\pi_1(\widehat{\mathscr G}_X, \widehat{\Gamma}_X,  \widehat{\mathsf T}_X)$ of the graph of groups $(\widehat{\mathscr G}_X, \widehat{\Gamma}_X)$ at $\widehat{\mathsf T}_X$ (see \cite{bass1976remarks}, \cite{serre1977arbres} and \cite{serre1980trees} section \S I.5), which is the group generated by the vertex groups in $\widehat{\mathscr G}_X$, with their own relations, together with a set of elements $\{g_{\mathsf e}\}_{\mathsf e\in\mathsf E(\widehat{\Gamma}_X)}$ verifying the following relations:
$$g_{\mathsf e} h^{\mathsf e}g_{\mathsf e}^{-1}=h^{\overline{\mathsf e}}\,\mbox{ for each }\,h\in G_{\mathsf e}\,;\quad\quad  g_{\mathsf e}^{-1}=g_{\overline{\mathsf e}}\,;\quad\quad g_{\mathsf e}=1 \mbox{ if } \mathsf e\in\mathsf E(\widehat{\mathsf T}_X)$$
where we denote by $h^{\mathsf e}$ the image of $h\in G_{\mathsf e}$ into $G_{t(\mathsf e)}$.
It turns out (\cite{serre1980trees}) that the isomorphism class of $\pi_1(\widehat{\mathscr G}_X, \widehat{\Gamma}_X, \widehat{\mathsf T}_X)$ does not depend from the choice of a maximal tree in $\widehat{\Gamma}_X$, hence we are allowed to speak of \textit{fundamental group of $(\widehat{\mathscr G}_X, \widehat{\Gamma}_X)$}, which in this case (by construction) coincides with the fundamental group of the irreducible $3$-manifold $X$.\\

 \noindent\textit{The Bass-Serre tree of the peripheral extension of the JSJ-splitting.} We shall now describe the Bass-Serre tree of the peripheral extension of the JSJ-splitting. The tree can be constructed from the data $(\widehat{\mathscr G}_X, \widehat{\Gamma}_X, \widehat{\mathsf T}_X)$. 
 \noindent In the sequel we implicitely assume to have fixed the isomorphism $\pi_1(X)\cong\pi_1(\widehat{\mathscr G}_X, \widehat{\Gamma}_X)$, so that we shall identify $G_{\mathsf v_k}$ with $\pi_1(X_k)$ and $G_{\mathsf e_{(k,i)}^{\bdr}}$ with $\pi_1(T_{(k,i)}^{\bdr})$. 
 The Bass-Serre tree $\mathcal T_X$ (where we omitted the dependence on $(\widehat{\mathscr G}_X, \widehat{\Gamma}_X,\widehat{\mathsf T}_X)$)  is defined as follows:
\begin{itemize}
	\item $\mathsf V(\mathcal T_X)=\bigsqcup_{\mathsf v\in\mathsf V(\widehat{\Gamma}_{X})}\pi_1(X)/ G_{\mathsf v}$\,;
	\item $\mathsf E(\mathcal T_X)=\bigsqcup_{\mathsf e\in\mathsf E(\widehat{\Gamma}_X)} \pi_1(X)/ G_{\mathsf w}^{\mathsf w}$ where $\mathsf w=\overline{|\mathsf e|}$ and $G_{\mathsf w}^{\mathsf w}$ denotes the subgroup of $G_{t(\mathsf w)}$ corresponding to $\mathsf w$.
\end{itemize}
We call $\tilde{\mathsf v}$ the coset in $\pi_1(X)/G_{\mathsf v}$ containing  $1$ and $\tilde{\mathsf e}$ the coset containing $1$ in $\pi_1(X)/G_{\mathsf w}^{\mathsf w}$ (among the cosets associated to $\mathsf e$). We have to define source, target and inverse of each edge. Let $\chi$ be the characteristic function of $\mathsf E_+(\widehat{\Gamma}_X)$. Then we define: 
$s(g\,\tilde{\mathsf e})=gg_{\mathsf e}^{\chi(\mathsf e)-1}\,\widetilde{s(\mathsf e)}\;;\quad t(g\,\tilde{\mathsf e})=gg_{\mathsf e}^{\chi(\mathsf e)}\widetilde{t(\mathsf e)}\;;\quad\overline{g\,\tilde{\mathsf e}}=g\,\tilde{\overline{\mathsf e}}\;.$
\normalsize

\begin{rmk}By construction, the tree $\mathcal T_X$ verifies the following:
\begin{enumerate}
 %\item two vertices  of $\mathcal T_X$ are adjacent if and only if  one of the following holds: either $\mathsf v=g G_{\mathsf v_{k_1}}$ and $\mathsf v'= g' G_{\mathsf v_{k_1}}$, with $n_{(k_1,k_2)}\neq 0$ and  $g^{-1}g'\in G_{\mathsf v_{k_1}}\cup G_{\mathsf v_{k_2}}$ or $\mathsf v=h G_{\mathsf v_{(k,i)}^{\bdr}}$ and $\mathsf v'= h' G_{\mathsf v_k}$, with $h^{-1}h'\in G_{\mathsf v_k}$.
  \item the vertices $g\,G_{\mathsf v_{(k,i)}^{\bdr}}$ are the (only) leaves of the Bass-Serre tree $\mathcal T_{X}$;
  \item  $\mathsf v=g\,G_{\mathsf v_{(k,i)}^{\bdr}}$ is adjacent to $\mathsf v'=g' G_{\mathsf v_{j}}$ if and only if $j=k$ and $g^{-1}g'\in G_{\mathsf v_{k}}$.
% \item let $g''\tilde{\mathsf e}$ be an edge of $\mathcal T_X$ and let $s(g''\tilde{\mathsf e})=g G_{\mathsf v_k}$ and $t(g''\tilde{\mathsf e})=g' G_{\mathsf v_j}$, then $g^{-1}g''\in G_{\mathsf v_k}$, $(g')^{-1}g''\in G_{\mathsf v_j}$.
\end{enumerate}
\end{rmk}

 The group $\pi_1(X)$ acts on the vertices and the edges of $\mathcal T_X$ by  left multiplication (as they are left cosets of subgroups of $\pi_1(X)$). This action turns out to be  an action  by automorphisms (without edge inversions) of  $\mathcal T_X$. By definition of $\mathcal T_X$ and of the action of $\pi_1(X)$ on $\mathcal T_X$ the stabilizer of the vertex  $g G_{\mathsf v}\in\mathsf V(\mathcal T_X)$  is the subgroup $\Stab_{\pi_1(X)}(g\,G_{\mathsf v})= g G_{\mathsf v}g^{-1}$. The same holds for the stabilizers of the edges. \\
 
 \textit{Glossary.} We shall distinguish four families of  vertices in $\mathcal T_X$:
\begin{itemize}
	\item  \textit{peripheral vertices}, \textit{i.e.} $h\, G_{\mathsf v_{k,i}^{\bdr}}$;
	\item  \textit{hyperbolic vertices}, \textit{i.e.} $h\, G_{\mathsf v_k}$ with $X_k$ a JSJ-component of hyperbolic type;
	\item \textit{Seifert fibered vertices}, \textit{i.e.} $h\, G_{\mathsf v_k}$ with $X_k$ a JSJ-component which is Seifert fibered with hyperbolic base orbifold;
	\item \textit{vertices of type $K\widetilde{\times}I$}, \textit{i.e.} $h\, G_{\mathsf v_k}$ with $X_k\simeq K\widetilde{\times}I$.
\end{itemize}

\subsection{The intersection of the leaves stabilizers}  We use the notation  introduced in the previous section. Before stating and proving Proposition \ref{int_leaves_stab} we need to establish two easy facts.

\begin{lem}\label{hyperbolic_vertex}
	Let $g\in\Stab_{\pi_1(X)}( hG_{\mathsf v_k})$, and assume that $\mathsf v_k$ is a hyperbolic vertex. Then $g$  can stabilize at most a single oriented edge having $h G_{\mathsf v_k}$ as source. 
\end{lem}

\begin{proof} Up to act on $\mathcal T_X$ by left multiplication with $h^{-1}$, we may assume that $h=1$. By construction of $\mathcal T_X$ the (oriented) edges having $G_{\mathsf v_k}=\pi_1(X_k)$ as source are either of type $b G_{\mathsf e_{(k,\ell)}^{j}}$ or $b G_{\mathsf e_{(k,i)}^{\bdr}}$, where $b\in\pi_1(X_k)$, which correspond to the left cosets in $\pi_1(X_k)$ of the peripheral subgroups corresponding to the boundary tori $T_{(k,\ell)}^{j}$, $T_{(s,k)}^r$, $T_{(k,i)}^{\bdr}$ of $X_k$. Hence the stabilizers of the edges starting from $G_{\mathsf v_k}$ corresponds to suitable conjugates $b\,\pi_1(T)\,b^{-1}<\pi_1(X_k)$ for some $T$ in $\{T_{(k,i)}^{\bdr}, T_{(s,k)}^r, T_{(k, \ell)}^j\}$  and the conclusion follows applying Lemma \ref{Hyperbolic} and conjugating by $h$.
\end{proof}

\begin{lem}\label{kappaT}
	Any geodesic path connecting two leaves of $\mathcal T_X$ and containing a vertex $h G_{\mathsf v_k}$ of type $K\widetilde{\times}I$ has length greater or equal to $4$. Moreover, a vertex of type $K\widetilde{\times}I$ is source of precisely two oriented edges, whose targets are two distinct left cosets associated to the JSJ-component adjacent to $X_k$.
\end{lem}

\begin{proof}
	A vertex  of type $K\widetilde{\times}I$ cannot be adjacent to a leaf of $\mathcal T_X$, because no boundary component of $X$ belong to a JSJ-component homeomorphic to $K\widetilde{\times}I$. This is sufficient to show that a geodesic path connecting two leaves of $\mathcal T_X$ and containing a vertex of type $K\widetilde{\times}I$ must have at least $5$ vertices, and thus must have length greater or equal to $4$.
	For what concerns the second claim observe first that the edge in $(\widehat{\Gamma}_X,\widehat{\mathscr G}_X)$ connecting $X_k$ to the adjacent JSJ-component, say $X_j$, is always contained in any maximal tree $\widehat T_X$. On the other hand, by Lemma \ref{kappatwistato} the peripheral subgroup $\pi_1(\bdr X_k)$ has index two in $\pi_1(X_k)$ and that the two left cosets of $\pi_1(\bdr X_k)$ in $\pi_1(X_k)$, using the notation of Lemma \ref{kappatwistato}, are $\pi_1(\bdr X_k)$ and $a\,\pi_1(\bdr X_k)$. Combining the previous informations we deduce that we have  precisely two oriented edges starting from $h\,\pi_1(X_k)$ ---\,namely $h\,\pi_1(\bdr X_k)= h\, G_{\mathsf v_{(k,j)}^{1}}$ and $ha\,\pi_1(\bdr X_k)=ha\,G_{\mathsf v_{(k,j)}^{1}}$--- whose targets are $h\, G_{\mathsf v_j}$ and $ha\,G_{\mathsf v_j}$ respectively. 
\end{proof}

We are now ready to prove:

\begin{prop}[Intersection of leaves stabilizers]\label{int_leaves_stab}
	 Let $X$ be an irreducible $3$-manifold with non-empty boundary, not homeomorphic to $K\widetilde{\times}I$, $T^2\times I$ or $D^2\times S^1$. Consider two distinct leaves $h_1 G_{\mathsf v_{(k_1,i_1)}^{\partial}}$, $h_2 G_{\mathsf v_{(k_2,i_2)}^{\partial}}$ of $\mathcal T_X$. Then
	%Assume that either $(k_1, i_1)\neq (k_2, i_2)$ or, in case  $k=k_1=k_2$ and $i=i_1=i_2$,  $h_1h_2\not\in\Stab_{\pi_1(X)}(G_{\mathsf v_{(k,i)}^{\bdr}})$, then
	\small
	$$\Stab_{\pi_1(X)}(h_1G_{\mathsf v_{(k_1,i_1)}^{\bdr}})\cap\Stab_{\pi_1(X)}(h_2 G_{\mathsf v_{(k_2, i_2)}^{\bdr}})\neq\{1\}$$
	\normalsize
	if and only if the following conditions hold:
	\begin{itemize}
		\item[(i)] $k=k_1=k_2$\,;
		\item[(ii)] $X_k$ is a Seifert fibered JSJ-component not homeomorphic to $K\widetilde{\times}I$\,;
		\item[(iii)] $h_1^{-1}h_2\in G_{\mathsf v_k}$\,.
	\end{itemize}
		Moreover, if  $hG_{\mathsf v_k}$ is the vertex of $\mathcal T_{X}$ adjacent to $h_{j} G_{\mathsf v_{(k,i_j)}^{\bdr}}$ ($j=1,2$) the previous intersection coincides with the infinite cyclic subgroup $h\langle f_k \rangle h^{-1}$, where $f_k$ is the regular fiber of $G_{\mathsf v_k}\cong\pi_1(X_k)$.\\
	\end{prop}

\begin{proof}
We may assume that the JSJ-decomposition of $X$ is non-trivial, otherwise the result follows trivially from the description of the Bass-Serre tree,  Lemma \ref{Seifert} and Lemma \ref{Hyperbolic}.
 Let $g\in\pi_1(X)$ be a non-trivial element and assume that
	\small
	$$g\in\Stab_{\pi_1(X)}(h_1G_{\mathsf v_{(k_1, i_1)}^{\bdr}})\cap\Stab_{\pi_1(X)}(h_2 G_{\mathsf v_{(k_2,i_2)}^{\bdr}})$$
	\normalsize
Since the fixed point set of an element of a group acting on a tree is a connected subtree, we deduce that $\fix_{\mathcal T_X}(g)$ contains the geodesic segment $\gamma:[0,M]\f\mathcal T_X$ joining the vertex $h_1 G_{\mathsf v_{(k_1,i_1)}^{\bdr}}$ and the vertex $h_2 G_{\mathsf v_{(k_2,i_2)}^{\bdr}}$. Here we have chosen the \textit{simplicial distance} on $\mathcal T_X$ --- where the edges of $\mathcal T_X$ are all isometric to $[0,1]\subset\mathbb R$ --- so that $\gamma(i)$ is a vertex of $\mathcal T_X$ for each integer $i=0,..., M$.
By construction of $\mathcal T_X$ the path between two distinct leaves necessarily has length greater or equal to $2$, hence  $M\ge 2$. 
%If $M$ is equal to $2$, the two leaves are both adjacent to a single vertex of $\mathcal T_X$, corresponding to some left coset of one of the vertices of the JSJ-splitting, which necessarily either corresponds to a JSJ-component  the conclusion follows by applying Lemma \ref{Seifert} and Lemma \ref{Hyperbolic}
We recall that the vertices $hG_{\mathsf v_{(k,i)}^{\bdr}}$ are leaves of the Bass-Serre tree $\mathcal T_X$.\linebreak It follows that the vertices contained into $\gamma((0,M))$ are necessarily of  type $h G_{\mathsf v_k}$ for a suitable $h\in \pi_1(X)$, because a geodesic segment cannot contain backtrackings.
We know that the interior of the path $\gamma$ does  contain neither hyperbolic vertices (by Lemma \ref{hyperbolic_vertex}), nor peripheral vertices. Now we shall prove that the length of the geodesic $\gamma$ is necessarily equal to $2$.

\begin{lem}\label{M=2} $M=2$
	\end{lem}

\begin{proof} Let $\gamma(i)=g_iG_{\mathsf v_{k_i}}$ be the vertices in  $\gamma((0,M))$.
We  know that the vertices $\gamma(i)$, $i\in\{1,..,M-1\}$ are either Seifert fibered vertices or vertices of type $K\widetilde{\times}I$.
Consider $\gamma(1)=g_1 G_{\mathsf v_{k_1}}$. By Lemma \ref{kappaT} we know that this vertex is a Seifert fibered vertex. Since $g$ stabilizes $\gamma$, it should stabilize two distinct (oriented) edges having $\gamma(1)$ as source and $\gamma(0)$, $\gamma(2)$ as target. By Lemma \ref{Seifert} we know that in this case
$g\in g_1\langle f_{k_1}\rangle g_1^{-1}<\Stab_{\pi_1(X)}(g_1G_{\mathsf v_{k_1}})$
 where $f_{k_1}$ is the regular fiber of $\pi_1(X_{k_1})$. We deduce  that either $\gamma(2)$ is a peripheral vertex or  by  Lemma \ref{Gluingiso} $g\in \Stab_{\pi_1(X)}(g_2 G_{\mathsf v_{k_2}})\smallsetminus \langle g_2 f_{k_2} g_2^{-1}\rangle$, for any choice of a regular fiber $f_{k_2}$ in $\pi_1(X_{k_2})$. Our aim is to show that $\gamma(2)$ is a peripheral vertex. Thus we need to prove that $\gamma(2)$ is neither a Seifert fibered vertex nor a vertex of type $K\widetilde{\times}I$.
 
 Assume first that $\gamma(2)$ is a Seifert fibered vertex. Since $g\not\in g_2\langle f_{k_2}\rangle g_2^{-1}$ it follows from Lemma \ref{Seifert} that
 $g$ does not stabilize edges other than the one connecting $g_1 G_{\mathsf v_{k_1}}$ to $g_2 G_{\mathsf v_{k_2}}$. In particular $g$ would not stabilize any geodesic path of length greater or equal to $3$. Since we are assuming the length of the geodesic path greater or equal to $3$, we need to exclude that $\gamma(2)$ is a Seifert fibered vertex.

Now consider the case where  $\gamma(2)$ is a vertex of type $K\widetilde{\times} I$. By Lemma \ref{kappaT} in this case $M\ge 4$ and  there are only two oriented edges starting from $g_2 G_{\mathsf v_{k_2}}$, which  are  those linking $\gamma(2)$ to the vertices $\gamma(1)=g_1 G_{\mathsf v_{k_1}}$ and $\gamma(3)=g_3G_{\mathsf v_{k_3}}=g_3 G_{\mathsf v_{k_1}}$ %Since $g\in g_2(\pi_1(T_{(k_1,k_2)})) g_2^{-1}$ and $\pi_1(T_{(k_1, k_2)})\unlhd\pi_1(X_{k_2})$ is a normal subgroup of index $2$, we deduce that  $g$ stabilizes $g_2 G_{\mathsf v_{k_2}}$ as well as both the edges having $\gamma(2)$ as source / target.\\
 %As a JSJ-component homeomorphic to $K\widetilde{\times} I$ gives rise to a leaf of the JSJ-splitting, the edge corresponding to the JSJ-torus connecting such a JSJ-component of  to the  adjacent JSJ-component is necessarily contained in the maximal tree $\widehat{\mathsf T}_X\subset\widehat{\Gamma}_X$.  
 and from Lemma \ref{kappaT} we deduce that $g_3\, G_{\mathsf v_{k_1}}=g_2a\, G_{\mathsf v_{k_1}}$. Since $g\in g_1\langle f_{k_1}\rangle g_1^{-1}$ we have  $g\in g_2\langle\psi_{(k_1,k_2;1)}^{\varepsilon}(f_{k_1})\rangle g_2^{-1}$, where we have to choose $\varepsilon\in\{\pm1\}$ depending on the orientation of the edge. By Lemma \ref{Gluingiso} we know that $\psi_{(k_1,k_2;1)}^{\varepsilon}(f_{k_1})\not\in\langle a^2\rangle\cup\langle f\rangle$. %Let $\psi_{(k_1,k_2;j_1)}^{\varepsilon}(f_{k_1})= a^{2k} f^{\ell}$; since $f_{k_1}$ is a primitive element in the boundary torus it follows that $\mathrm{GCD}(k,\ell)=1$. We have $g=g_2\,a^{2p} f^{q}\,g_2^{-1}$. 
 %On the other hand, as $g$ stabilizes the edge connecting $\gamma(2)$ to $\gamma(3)$, we know that $g\in\Stab_{\pi_1(X)}(g_2a G_{\mathsf e_{(k_1,k_2)}^{j}})$. Hence 
 We can write $g$  as:
 \small
$$g=g_2 \psi_{(k_1, k_2;1)}^{\varepsilon}(f_{k_1}^{m})g_2^{-1}=(g_2a) (a^{-1}\psi_{(k_1, k_2;1)}^{\varepsilon}(f_{k_1}^{m}) a)\, (g_2a)^{-1}$$
\normalsize
Since $\psi_{(k_1, k_2;1)}^{\varepsilon}(f_{k_1})\not\in\langle a^2\rangle\cup\langle f\rangle$, it follows from Lemma \ref{kappatwistato} that  
\small
$$\langle a^{-1}\psi_{(k_1, k_2;1)}^{\varepsilon}(f_{k_1}) a\rangle\cap\langle\psi_{(k_1, k_2;1)}^{\varepsilon}(f_{k_1})\rangle=\{1\}$$
\normalsize
 Thus, looking at $g$ as an element of $\Stab_{\pi_1(X)}(g_3 G_{\mathsf v_{k_1}})=\Stab_{\pi_1(X)}(g_2a\, G_{\mathsf v_{k_1}})$, we obtain the expression  $g=(g_2a\,)\psi_{(k_1,k_2;1)}^{-\varepsilon}\left(a^{-1}\psi_{(k_1,k_2; 1)}^{\varepsilon} (f_{k_1}^m) a\right)\,(g_2a)^{-1}$. From the previous discussion it follows that $g\not\in (g_2a)\langle f_{k_1}\rangle (g_2a)^{-1}$ and hence  $g\not\in g_3\langle f_{k_1}\rangle g_3^{-1}$ (as $g_3^{-1}g_2a\in G_{\mathsf v_{k_1}}$ and $\langle f_{k_1}\rangle$ is normal in $ G_{\mathsf v_{k_1}}$). Using Lemma \ref{Seifert} we conclude that $g$ does not stabilize the edges starting from $\gamma(3)$ except for the one connecting $\gamma(3)$ to $\gamma(2)$. But this is a contradiction, since by Lemma \ref{kappaT} $g$ should stabilize a segment $\gamma$ of length greater or equal to $4$.
 
 Since $\gamma(2)$ is neither a hyperbolic vertex, nor a Seifert fibered vertex, nor a vertex of type $K\widetilde{\times}I$ we conclude that $\gamma(2)$ is a peripheral vertex, \textit{i.e.} $M=2$ and $\gamma(2)=h_2 G_{\mathsf v_{(k_2,i_2)}^{\bdr}}$.
\end{proof}

\noindent \textit{End of the Proof of Proposition \ref{int_leaves_stab}.} By  Lemma \ref{M=2}, we know that if  
\small
$$\Stab_{\pi_1(X)}(h_1 G_{\mathsf v_{(k_1,i_1)}^{\partial}})\cap\Stab_{\pi_1(X)}(h_2 G_{\mathsf v_{(k_2, i_2)}^{\bdr}})\neq\{1\}$$
\normalsize
then the path connecting $h_1 G_{\mathsf v_{(k_1, i_1)}^{\partial}}$ to $h_2 G_{\mathsf v_{(k_2,i_2)}^{\partial}}$ has length equal to $2$. This means that both the leaves are adjacent to a single vertex $h G_{\mathsf v_k}$, which is necessarily a Seifert fibered vertex by Lemma \ref{hyperbolic_vertex} and Lemma \ref{kappaT}. By construction of $\mathcal T_X$ we deduce that $k_1=k_2=k$, and  $h^{-1}h_i\in G_{\mathsf v_{k}}$ for $i=1, 2$, and thus $h_1^{-1}h_2\in G_{\mathsf v_k}$.\\ Now assume that $i_1\neq i_2$. The stabilizers of the leaves are precisely \small $$h_{j} G_{\mathsf v_{(k,i_j)}^{\partial}} h_{j}^{-1}=h_j\pi_1(T_{(k,i_j)}^{\bdr})h_j^{-1}\quad\quad j=1,2$$ \normalsize  Conjugating both subgroups by  $h_1$, the problem is reduced to find the intersection of the subgroups $\pi_1(T_{(k, i_1)}^{\partial})$ and $(h_1^{-1}h_2)\,\pi_1(T_{(k, i_2)}^{\partial})\,(h_2^{-1}h_1)$ in $\pi_1(X_k)$. It follows from Lemma \ref{Seifert} that this intersection is equal to the subgroup $\langle f_k\rangle$. The same conclusion holds (again in view of Lemma \ref{Seifert}) if $i=i_1=i_2$ and $h_2^{-1}h_1\not\in\pi_1(T_{(k,i)}^{\partial})$.
\end{proof}

\begin{proof}[Proof of Theorem \ref{AI_bdrtori} and  of \cite{delaHarpe2014onmalnormal} Theorem 3]
	Since $h_j G_{\mathsf v_{(k_j, i_j)}^{\partial}}$ are two distinct leaves of $\mathcal T_X$, then  either \small $(k_1,i_1)\neq(k_2, i_2)$\normalsize, or  \small$(k,i)=(k_1,i_1)=(k_2,i_2)$\normalsize\, but $h_1^{-1}h_2\not\in G_{\mathsf v_{(k,i)}^{\partial}}$.
	Theorem \ref{AI_bdrtori} (i) is equivalent to  Proposition \ref{int_leaves_stab} case $ (k,i)=(k_1,i_1)=(k_2,i_2)$ and $h_1^{-1}h_2\not\in\pi_1(T_{(k,i)}^{\partial})$, whereas Theorem \ref{AI_bdrtori} (ii) corresponds to Proposition \ref{int_leaves_stab} case $(k_1,i_1)\neq(k_2,i_2)$.  Proposition \ref{int_leaves_stab} (ii) is precisely the content of \cite{delaHarpe2014onmalnormal} Theorem 3. Finally, we observe that Proposition \ref{int_leaves_stab} implies that abelian peripheral subgroups corresponding to the boundary tori which belong to distinct JSJ-components are conjugately separated.
	\end{proof}

\section{Malnormal splittings of $3$-manifold groups}

\subsection{Malnormal amalgamated products and HNN-extensions}

\subsubsection{Definitions}

In this section we shall first recall the definition of \textit{$k$-step malnormal amalgamated product} (see \cite{karras1971malnormal}) and we shall give the natural analogous, the notion of \textit{$k$-step malnormal HNN-extension}. We shall recall their relation with the notion of $k$-\textit{acylindrical splitting}  (\cite{sela1997acylindrical}) in the  setting of Bass-Serre theory. For a general introduction to combinatorial group theory we refer to the textbooks \cite{lyndon1977combinatorial}, \cite{magnus2004combinatorial}; see also section \S1 in \cite{scott1979topological}.\\

\noindent Let $G$ be a group. Let $H<G$ be a subgroup of $G$. We recall that $H$ is \textit{malnormal} in $G$ if $gHg^{-1}\cap H\neq\{1\}$ implies $g\in H$. 

\begin{defn}[Extended normalizer]
	Let $H$ be a subgroup of a discrete group $G$ and let $h\in H$. The \textit{extended normalizer} $\mathcal E_G^H(h)$ of $h$ relative to $H$ in $G$ is the set: ${\mathcal E}_G^H(h)=\{g\in G\,|\; ghg^{-1}\in H\}$, when $h\neq 1$, and $\mathcal E_G^H(1)= H$. We define the extended normalizer of $H$ in $G$ by $\mathcal E_G(H)=\bigcup_{h\in H} \mathcal E_G^H(h)$.
\end{defn}

\noindent In \cite{karras1971malnormal} Karrass and Solitar introduced the notion of $k$-step malnormal amalgamated product (or, for short, $k$-step malnormal product) combining the notion of malnormality of a subgroup and the existence of a \textit{normal form} for the elements of an amalgamated product (and thus of a notion of \lq\lq length").
Let us denote the \textit{syllable length} of the element $g$ by $|\,g\,|$. If $K$ is a subset of $G$ we shall denote $\mathscr L(K)=\max_{g\in K} |\,g\,|$.

\begin{defn}[$k$-step malnormal product]
	Let $G\cong A\,_C^*B$ be an amalgamated product. The subgroup $C$ is $k$-step malnormal in $G$ if and only if $\mathscr L(\mathcal E_G(C))\le k$. We shall say that $G$ is a \textit{$k$-step malnormal product}. It is easily seen that $G$ is a $0$-step malnormal product if and only if $C$ is malnormal in both $A$ and $B$.
\end{defn}

%\noindent In the sequel we shall need the following definitions.\\

%\noindent Let $G=A\,_C^*B$; we shall say that $g\in G$ is \textit{cyclically reduced} if either $g\in A, B$ or $|\,g\,|\ge 2$ and the multi-index $\underline X_g$ associated with $g$ is such that $X_1\neq X_n$.\\

%\noindent Some useful property which is satisfied by cyclically reduced elements.
%\begin{itemize}
%\item[(P1)] any element $g\in G$ is conjugate to a cyclically reduced element\,; 
%\item[(P2)] any cyclically reduced element $g\in G$ such that $|\,g\,|\ge 2$ is of infinite order\,;
%\item[(P3)] let $g\in G$ be a cyclically reduced element of syllable length $|\,g\,|\ge 2$, then $|\,g^k\,|=k\cdot |\,g\,|$, and the alternating multi-index associated with $g^k$ is $\underbrace{(\underline X,..., \underline X)}_{k-times}$, where $\underline X$ is the alternating multi-index associated with $g$.\\
%\end{itemize}

%\noindent Let $g\in G=A\,_C^*B$; the  \textit{cyclically reduced length} of $g$, $|\,g\,|_{\circlearrowright}$ is the syllable length of any cyclically reduced element which is conjugate to $g$. This number does not depend on the particular conjugate of $g$ provided it is cyclically reduced.\\

\noindent The HNN-extension of a group $A$ via an isomorphism $\varphi: C_{-1}\f C_{1}$ between two of its subgroups, $C_{-1}$, $C_{1}$, is defined as the group having the following presentation:
\small
$$A\,^\ast_{\varphi}=\langle A, t\;|\; \mbox{rel}(A);\, tc t^{-1}=\varphi(c), \forall c\in C_{-1}\rangle$$
\normalsize 
\noindent where  $\mathrm{rel}(A)$ are the relations  of the group $A$. We remark that any element $g$ in $G\cong A\,^\ast_\varphi$ can be written as:
$g=w_0\,t^{\varepsilon_1} w_1\, t^{\varepsilon_2}\cdots t^{\varepsilon_n} \,w_n$
where $\varepsilon_i=\pm 1$ and $w_i\in A$.
\vspace{2mm}

\noindent \textbf{Britton's Lemma}.\emph{
	Let $G\cong A\,^\ast_\varphi$ and let $G\ni g= w_0\,t^{\varepsilon_1}\,w_1\cdots t^{\varepsilon_n}\,w_n$, as before. If  $g = 1$ one of the two following conditions holds:
	\begin{itemize}
		\item[(i)] $n=0$ and $w_0= 1$\,;
		\item[(ii)] $g$ contains either $t w_i t^{-1}$ with $w_i\in C_{-1}$ or $t^{-1} w_i t$ with $w_i\in C_{1}$. For future reference we shall say that such a form has pinches.
	\end{itemize}
}

\noindent We shall say that $g=w_0\,t^{\varepsilon_1} w_1\, t^{\varepsilon_2}\cdots t^{\varepsilon_n} \,w_n$ is a reduced form for $g$ if it contains no pinches. A consequence of  Britton's Lemma is that any element of $G= A\,^\ast_\varphi$ has a reduced form. The number of occurrences of the stable letter $t$, raised to the power $\pm 1$ in a reduced form for $g$ does not change if we change the reduced form. Hence we define the length of the element $g$ as  the number $|\,g\,|$ of occurrences of the stable letter raised to the power $\pm 1$  in a reduced writing for $g$.

\begin{defn}
	Let us consider $A\,_\varphi^*$ where $\varphi :C_{-1}\f C_{1}$ is the isomorphism. We shall say that $G\cong A\,^\ast_{\varphi}$ is a \textit{$k$-step malnormal HNN-extension} if $g\,C_{\varepsilon}\,g^{-1}\cap C_{\pm\varepsilon}\neq 1$ for $\varepsilon=\pm 1$, implies $|\,g\,|\le k$.
\end{defn}

\subsubsection{Malnormality and acilindricity}
\noindent Let $G$ be the fundamental group of the  graph of groups $(\mathscr G,\Gamma)$  defined as follows: the underlying graph has two vertices $\{\mathsf v_{A}, \mathsf v_{B}\}$ and two (oriented) edges $\{\mathsf e_{C}, \overline{\mathsf e}_C\}$, with $\overline{\mathsf e_C}=\overline{\mathsf e}_C$, $s(\mathsf e_C)=\mathsf v_A$ and $t(\mathsf e_C)=\mathsf v_B$; the collection of vertex groups is given by $A\cong G_{\mathsf v_A}$, $B\cong G_{\mathsf v_B}$ and the edge group is $C\cong G_{\mathsf e_C}\cong G_{\overline{\mathsf e}_C}$ with monomorphisms given by $\iota_A: C\hookrightarrow A$ and $\iota_B:C\hookrightarrow B$.  Hence $G\cong\pi_1(\mathscr G,\Gamma)\cong A\,_C^*B$.

\begin{prop}
	The splitting $(\mathscr G, \Gamma)$ is $k$-acylindrical if and only if $\pi_1(\mathscr G,\Gamma)\cong A\,_C^*B$ is a $(k-1)$-step malnormal product.
\end{prop}

\begin{proof} Consider the Bass-Serre tree  $\mathcal T_{(\mathscr G, \Gamma)}$ associated to the graph of groups previously described and to the choice of $\{\mathsf e_C\}$ as orientation (no choice of a maximal tree is needed in this case, since $\Gamma$ is a tree). %is the tree where the set of vertices is given by $\mathsf V(\mathcal T_{(\mathscr G, \Gamma)})=\{gA\,|\,g\in G\}\sqcup\{gB\,|\, g\in G\}$, \textit{i. e. } the left cosets of $A$ and $B$ in $G$, the set of edges of $\mathcal T_{(\mathscr G, \Gamma)}$ is  $\mathsf E(\mathcal T_{(\mathscr G, \Gamma)})=\sqcup_{1,2}\{g\iota_A(C)\,|\, g\in G\}$, with source, target, orientations defined as we explained in section \S3.1. 
We are assuming here to have chosen the representative of the left coset $hA$ (resp. $hB$) so that $|ha|>|h|$ for any $a\in A\smallsetminus C$ (resp. $|hb|>|h|$ for any $b\in B\smallsetminus C$). By construction two vertices $gA$ and $hB$ are adjacent if and only if  either there exists $b\in B\smallsetminus C$ such that $g=hb$ or there exists $a\in A\smallsetminus C$ such that $h=ga$ or $g=h=1$. 

\noindent Consider the action of $G\cong A\,^*_CB$ on $\mathcal T_{(\mathscr G, \Gamma)}$. The stabilizers of the vertices are given by 
$\Stab_{G}(gA)\cong gAg^{-1}$ and $\Stab_{G}(hB)\cong h Bh^{-1}$. Moreover the edge $\mathsf e$ between $gA$ and $hB$  is equal to $\mathsf e=\delta \iota_A(C)$ where $\delta$ is the longest element between $g$ and $h$. The stabilizer is $\Stab_{G}(\mathsf e)\cong\delta\,\iota_A(C)\, \delta^{-1}$. Let $g\in G$ be an element which stabilizes a set in $\mathcal T_{(\mathscr G, \Gamma)}$ having diameter  equal to $k'\in\N$. Consider a geodesic of length $k'$ between the two vertices $\mathsf v_1,\mathsf v_2$ realizing the diameter of  $\fix_{\mathcal T_{(\mathscr G, \Gamma)}}(g)$, the fixed point set of $g$. We shall assume that $\mathsf v_1=h_1A$ and $\mathsf v_2= h_2B$ (the other cases are similar).  Since the element $g$ stabilizes both $\mathsf v_1$ and $\mathsf v_2$ we have that $g\in \Stab_{G}(\mathsf v_1)\cap \Stab_{G}(\mathsf v_2)$. From the structure of $\mathcal T_{(\mathscr G,\Gamma)}$ and the previous choice of a representative system for the left cosets of $A$ and $B$, it is straightforward to check that $k'$ is either equal to $|h_1^{-1}h_2|+1$ or to $|h_1^{-1} h_2|$ (depending whether the path which leads from $1\, A$ to $(h_1^{-1} h_2)\, B$ pass through the edge $1\, C$ or no). Observe that
\small
$$g\in\Stab_{G}(\mathsf v_1)\cap \Stab_{G}(\mathsf v_2) \Leftrightarrow g\in h_1Ah_1^{-1}\cap h_2Bh_2^{-1}\Leftrightarrow g\in h_1 Ch_1^{-1}\cap h_2 Ch_2^{-1}$$
\normalsize
Assume that $G\cong A\,_C^*B$ is a $k$-step malnormal product; the latter intersection is non-trivial only if $k'-1\le|h_1^{-1}h_2|\le k$. A similar conclusion holds if we replace $h_2B$ with a vertex of type $hA$ or $h_1 A$ with a vertex of type $hB$. This implies that the action of $G$ on $\mathcal T_{(\mathscr G, \Gamma)}$ is $(k+1)$-acylindrical, \textit{i.e.} the graph of groups $(\mathscr G,\Gamma)$ is a $(k+1)$-acylindrical splitting of $G$.
 Viceversa, assume that the graph of groups $(\mathscr G,\Gamma)$ is a $(k+1)$-acylindrical splitting for $G\cong A\,_C^*B$. The intersection of the two stabilizers is non-trivial only if $|h_1^{-1}h_2|+1\le k+1$ \textit{i.e.} only if $C$ is $k$-step malnormal in $G \cong A\,_C^*B$.
 \end{proof}

Now let us consider an amalgamated product $G=A\,_\varphi^*$ where $\varphi: C_{-1}\f C_1$ is an isomorphism between two subgroups of $A$. We can construct the following graph of groups $(\mathscr G,\Gamma)$: consider a loop formed by a vertex $\mathsf v$ and one edge $\mathsf e$ and define $G_{\mathsf v}=A$ and $C=G_{\mathsf e}=G_{\bar{\mathsf e}}$ with isomorphisms $G_{\mathsf e}\f C_{-1}<A$, $G_{\bar{\mathsf e}}\f C_1<A$.

\begin{prop}
	The splitting $(\mathscr G,\Gamma)$ is $k$-acylindrical if and only if $\pi_1(\mathscr G,\Gamma)=A\,_\varphi^*$ is a $k$-step malnormal HNN-extension.
\end{prop}

\begin{proof}
	As in the previous case we consider the Bass-Serre tree  $\mathcal T_{(\mathscr G, \Gamma)}$ associated to the graph of groups previously described and to the choice of $\{\mathsf e_C\}$ as orientation (no choice of a maximal tree is needed in this case, since $\Gamma$ is a tree). The set of vertices of this tree is given by the left cosets $\{g\,A\}$; we choose representatives of the left cosets so that their reduced form ends with the stable letter to the power $\pm 1$. The set of edges similarly is given by the left cosets $\{g\, C_{-1}\}$. Let $g_1, g_2$ be two representative of left cosets of $A$ and let $|g_2|>|g_1|$ there exists an edge bewteen $g_1A$ and $g_2 A$ if and only if there exists an $a\in A$ such that either $g_2=g_1a t$ or $g_2=g_1at^{-1}$ and in that case the corresponding edge will be respectively either $g_1a C_1$ or $g_1a C_{-1}$. Now, assume that $A\,_\varphi^*$ is a $k$-step malnormal HNN-extension and take $g\in G$. Assume that $g$ fixes a path between $g_1\,A$ and $g_2\,A$. We assume that the length of this path is $|g_1^{-1}g_2|>k$ This means that $g\in\Stab_{G}(g_1 A g_1^{-1})\cap\Stab_{G}(g_2 Ag_2^{-1})$. By construction of the Bass-Serre tree this is possible if and only if $g\in g_1a C_{\varepsilon}a^{-1}g_1^{-1}\cap g_2a' C_{\delta} a'^{-1}g_2^{-1}$ for suitable $a, a'\in A$, where $\varepsilon$, $\delta\in\{\pm 1\}$ and should be chosen depending on the first and the last edge of the path. But this intersection is empty because the intersection $C_{\varepsilon}\cap (g_1^{-1}g_2)C_{\delta}(g_2^{-1}g_1)$ is empty by the $k$-step malnormality of the HNN-extension. Conversely, assume that the action of $G$ on $\mathcal T_{(\mathscr G,\Gamma)}$ is $k$-acylindrical, this means that if $g_1 A$, $g_2 A$ are at distance greater than $k$ then the respective stabilizers are disjoint. Since the following equality holds $\Stab_{G}(g_1A)\cap\Stab_{G}(g_2 A)= g_1a C_{\varepsilon}a^{-1}g_1^{-1}\cap g_2a' C_{\delta} a'^{-1}g_2^{-1}$ we see that the intersection on the right hand side is trivial, which implies that, whenever $|g|>k$ then $g C_{\varepsilon} g^{-1}\cap C_{\pm 1}=\{1\}$, which proves that the HNN-extension is $k$-step malnormal.
\end{proof}

% In a similar way it is possible to show that a $k$-step malnormal HNN-extension is a $k$-acylindrical splitting having a loop as underlying graph.

\subsection{Malnormal splittings.} This subsection is devoted to prove Proposition \ref{splittings}. Let us establish first the following preparatory Lemma:

\begin{lem}\label{tech}
	Let $X\neq T^2\times I$ be an irreducible $3$-manifold and let $\{T^{\pm1}\}\subseteq \bdr X$ be two distinct boundary tori. Let  $f:T^{-1}\f T^{+1}$ be a gluing giving rise to a JSJ-torus for $X'$,  the resulting irreducible $3$-manifold, and by $\varphi=f_*:\pi_1(T^{-1})\f \pi_1(T^{+1})$ the induced isomorphism. Then for any pair of elements 
	\small
	$$(g_-, g_+)\in[\pi_1(X)\smallsetminus\pi_1(T^{-1})]\times [\pi_1(X)\smallsetminus\pi_1(T^{+1})]$$
	\normalsize
	we have
	\small
	$$(g_+)\,\varphi\left((g_-)\pi_1(T^{-1})(g_-)^{-1}\cap\pi_1(T^{-1})\right)\,(g_+)^{-1}\cap\pi_1(T^{\pm1})=\{1\}$$
	$$(g_-)\,\varphi^{-1}\left((g_+)\pi_1(T^{+1}) (g_+)^{-1}\cap\pi_1(T^{+1})\right)\,(g_-)^{-1}\cap\pi_1(T^{\pm1})=\{1\}$$
	\normalsize
\end{lem}

\begin{proof} We shall denote by $X_{k_{\pm1}}$ the JSJ-components of $X$ which contain respectively $T^{\pm1}$ in their boundary (possibly $k_{-1}=k_{+1}$). We point out that $T^{-1}$ and $T^{+1}$ belong either to hyperbolic JSJ-components or to Seifert fibered JSJ-components with hyperbolic base orbifolds.
	We shall consider two different cases:
	\begin{itemize}
		\item[(1)]  at least one between $X_{k_{\pm1}}$ is a hyperbolic JSJ-component;
		\item[(2)] both $X_{k_{\pm1}}$ are Seifert fibered  JSJ-components.\\
	\end{itemize}
	
	\noindent\textit{Case} (1). Let us assume that  $T^{-1}$ is a boundary torus of a hyperbolic JSJ-component.  By \cite{delaHarpe2014onmalnormal} (see also Proposition \ref{int_leaves_stab}) we know that the abelian subgroups associated to the boundary tori of the hyperbolic JSJ-components are malnormal in the whole fundamental group, hence  for $g_-\in\pi_1(X)\smallsetminus\pi_1(T^{-1})$ we have $(g_-)\,\pi_1(T^{-1})\,(g_-)^{-1}\cap\pi_1(T^{-1})=\{1\}$, proving the first equality.\\ 
	To prove the second equality we remark that $\varphi^{-1}((g_+)\pi_1(T^{+1})(g_+)^{-1}\cap\pi_1(T^{+1}))$ is a (possibly trivial) subgroup of $\pi_1(T^{-1})$. The latter subgroup is malnormal in $\pi_1(X)$, hence a fortiori
		\small
	$$g_-\varphi^{-1}\left((g_+)\,\pi_1(T^{+1})\,(g_+)^{-1}\cap\pi_1(T^{+1})\right)g_-^{-1}\cap\pi_1(T^{-1})=\{1\}$$
	\normalsize 
	On the other hand, it follows from Proposition \ref{int_leaves_stab} that $\pi_1(T^{-1})$  is conjugately separated from $\pi_1(T^{+1})$. Hence: 
	\small
	$$g_-\varphi^{-1}\left((g_+)\,\pi_1(T^{+1})\,(g_+)^{-1}\cap\pi_1(T^{+1})\right)g_-^{-1}\cap\pi_1(T^{+1})=\{1\}$$
	\normalsize
	which gives the desired equality.\\
	
	\noindent\textit{Case} (2). Now suppose that $T^{-1}$ and $T^{+1}$ belong to two Seifert fibered JSJ-components (possibly the same JSJ-component).   Let $w\in\pi_1(T^{-1})$, by assumption  $g_-\in\pi_1(X)\smallsetminus\pi_1(T^{-1})$ hence by Theorem \ref{AI_bdrtori} (i) we see that $(g_-)\,w\,(g_-)^{-1}\in\pi_1(T^{-1})$ if and only if $g_-\in\pi_1(X_{k_{-1}})$ and $w\in\langle f_{k_{-1}}\rangle$. If this is the case then $(g_-)\,w\,(g_-)^{-1}=f_{k_{-1}}^{\ell}$ for some $\ell\in\Z$. Thanks to Lemma \ref{Gluingiso} we have that $\varphi(\langle g_-wg_-^{-1}\rangle)\cap\langle f_{k_{+1}}\rangle=\{1\}$. Applying  Theorem \ref{AI_bdrtori} (i) to the subgroup $\pi_1(T^{+1})$ we deduce that that $(g_+)\,\varphi( f_{k_{-1}}^{\ell})\,(g_+)^{-1}\not\in\pi_1\left( T^{+1}\right)$. Also, by Theorem \ref{AI_bdrtori} (ii), we have that $(g_+)\varphi(f_{k_{-1}}^\ell)(g_+)^{-1}\not\in\pi_1(T^{-1})$ and thus:
	\small
	$$(g_+)\varphi((g_-)\,\pi_1(T^{-1})\,(g_-)^{-1}\cap\pi_1(T^{-1}))(g_+)^{-1}\cap\pi_1(T^{\pm1})=\{1\}$$
	\normalsize
	The proof of the other equality is analogous to this one.
%	\noindent\textit{Case} (3). We assume that  $T^{\pm1}$ belong to the same Seifert fibered JSJ-component. Notice that, by Lemma \ref{Gluingiso}, $\varphi(\langle f_{k_-}\rangle)\cap\langle f_{k_+}\rangle=1$.\\
%	We consider a pair $(g_-, g_+)\in\left[\pi_1(X)\smallsetminus\pi_1(T^{-1})\right]\times\left[\pi_1(X)\smallsetminus\pi_1(T^{+1})\right]$.\\
%	We know that if $\gamma\in\pi_1(X)$, then $\gamma^{-1}\,w\,\gamma\in\pi_1(T^{-1})$ if and only if $\gamma^{-1}\,w\,\gamma=f_{k}^{\ell}$ for some $\ell\in\Z$. We apply $\varphi$ to the previous  element and we observe that, since $\varphi(\langle f_{k}\rangle)\cap\langle f_{k}\rangle=1$, it follows from property (C)  that $(\gamma')^{-1}\varphi_i( f_{k}^{\ell})\gamma'\not\in\pi_1\left( T^{+1}\right)$. On the other hand, by property (B) we have that $(\gamma')^{-1}\varphi(f_{k_{-}}^\ell)\gamma'\not\in\pi_1(T^{-1})$. We conclude that:
%	\small
%	$$(\gamma')^{-1}\varphi(\gamma^{-1}\,\pi_1(T^{-1})\,\gamma\cap\pi_1(T^{-1}))\gamma'\cap\pi_1(T^{\pm1})=1$$
%	\normalsize
%	The proof of the other equality is analogous.
\end{proof}

\begin{proof}[Proof of Proposition \ref{splittings}]
	Since we excluded $X$ to be finitely covered by a torus bundle over the circle we observe  that the different cases listed in the statement of Proposition \ref{splittings} cover any possible scenario. We shall now do a case by case analysis.\\
	
	\textit{Proof of} (D1). By assumption $X_{k_{-1}}$ and $X_{k_{+1}}$ are  hyperbolic JSJ-components. Let $X'$ and $X''$ be the closures of the two connected components of $X\smallsetminus T$. By \cite{delaHarpe2014onmalnormal}, Theorem 3,  $\pi_1(T^{-1})$ is malnormal in $\pi_1(X')$ and $\pi_1(T^{+1})$ is malnormal in $\pi_1(X'')$. An amalgamated product $A\,_C^*B$ where $C$ is malnormal both in $A$ and $B$ it is readily seen to be $0$-step malnormal. \\
	
	\textit{Proof of} (D2). Assume that $X_{k_{-1}}$ is a hyperbolic JSJ-component. Let $X'$ and $X''$ be the closures of the two connected components of $X\smallsetminus T$. By \cite{delaHarpe2014onmalnormal} Theorem 3, $\pi_1(T^{-1})$ is malnormal in $\pi_1(X')$. It is straightforward to check that  an amalgamated product $A\,_C^*B$ where $C$ malnormal in $A$ (or $B$) is a $1$-step malnormal amalgamated product.\\
	
	\textit{Proof of} (D3).  By assumption $X_{k_{-1}}$ and $X_{k_{+1}}$ are both Seifert fibered manifolds with hyperbolic base orbifolds. Let $X'$ and $X''$ be the closures of the two connected components of $X\smallsetminus T$. %By Theorem \ref{AI_bdrtori} (i) we know that if $g\in\pi_1(X')$ and  $g\pi_1(T^{-1})g^{-1}\cap\pi_1(T^{-1})\neq\{1\}$, then $g\in\pi_1(X_{k_{-1}})$ and the previous intersection is equal to the infinite cyclic subgroup $\langle f_{k_{-1}}\rangle$. The analogous statement holds for $\pi_1(T^{+1})$ in $\pi_1(X'')$.
	Let $g\in\pi_1(X)$ be such that $|g|\ge2$ and consider a reduced form $g=b_1a_1\cdots b_ma_m$ (with $b_1, a_m$ possibly equal to $1$). Let $c\in\pi_1(T)$ and take 
	$gcg^{-1}=(b_1\cdots a_m)\,c\,(b_1\cdots a_m)^{-1}$. Notice that, by definition of syllable length, if $|gcg^{-1}|>0$ then $gcg^{-1}\not\in\pi_1(T)$.
	Assume that $a_m\neq 1$ (the case $a_m=1$ is  analogous). By Theorem \ref{AI_bdrtori} (i) either $gcg^{-1}$ is reduced (and $|gcg^{-1}|\ge 3$) or $a_m\in\pi_1(X_{k_{-1}})$ and $c=f_{k_{-1}}^{\ell}$ for a suitable $\ell\in\Z$. In this second case by Lemma \ref{Gluingiso} we know that $\varphi(a_mca_m^{-1})=\varphi(f_{k_{-1}}^\ell)\not\in\langle f_{k_{+1}}\rangle$. Again by Theorem \ref{AI_bdrtori} (i) we have that $b_m\varphi(f_{k_{-1}})b_m^{-1}\not\in\pi_1(T^{+1})$ and hence we conclude that either $gcg^{-1}\in\pi_1(X'')\smallsetminus\pi_1(T^{+1})$  or $|g\,c\,g^{-1}|\ge 2$.  In both cases $gcg^{-1}\not\in\pi_1(T)$.\\
	
	\textit{Proof of} (D4). Without loss of generality we  assume that $X''\simeq K\widetilde{\times}I$. We recall  that $\pi_1(T^{+1})$ is normal in $\pi_1(X'')\cong \pi_1(K\widetilde{\times}I)$ (see Lemma \ref{kappatwistato}).
    %We  shall use the following notation:  $\pi_1(K\widetilde{\times}I)=\langle b, f\,|\; b^{-1}fb=f^{-1}\rangle$. The two regular fibers are $f$ and $b^2$. The subgroup $\pi_1(T^{+1})$ is identified with $\langle f, b^2\rangle$. 
	\noindent Let $g\in\pi_1(X)$ be such that $|g|\ge 4$. We shall show that $g\,c\,g^{-1}\not\in\pi_1(T)$ for any $c\in\pi_1(T)$; as in the previous case observe that if $|gcg^{-1}|>0$ then $gcg^{-1}\not\in\pi_1(T)$. Let  $g=b_1a_1\cdots b_ma_m$ be a reduced form for $g$, with $b_1, a_m$ possibly equal to $1$.  We treat separately  the case where $a_m\neq 1$ and the case where $a_m=1$.\\
	Assume $a_m\neq 1$. By Theorem \ref{AI_bdrtori} (i) we know that either
	\small $$g\,c\,g^{-1}=(b_1\cdots b_m)(a_m\,c\, a_m^{-1})(b_1\cdots b_m)^{-1}$$
	\normalsize
	 is a reduced form and $|gcg^{-1}|\ge 7$ or  $a_m\,c\,a_m^{-1}\in\pi_1(T^{-1})$, $c=f_{k_{-1}}^{\ell}$ for a suitable $\ell\in\Z$ and $a_m\in\pi_1(X_{k_{-1}})$. Thanks to  Lemma \ref{Gluingiso} we know that the element $\varphi(a_m c a_m^{-1})=\varphi(f_{k_{-1}}^{\pm\ell})$ does not belong to  the two infinite cyclic subgroups generated by the regular fibers of $X''\simeq K\widetilde{\times}I$ \footnote{Here and after the sign $\pm$ depends on the orientability of the base orbifold of the JSJ-component $X_{k_{-1}}$, and on the element $a_m\in\pi_1(X_{k_{-1}})$. See the presentation of the fundamental group of a Seifert fibered manifold in subsection \S2.1.1.}. By Lemma \ref{kappatwistato} we have  $b_m\varphi(f_{k_{-1}}^{\pm\ell})b_m^{-1}\not\in\langle\varphi(f_{k_{-1}})\rangle$. Hence the element $\varphi^{-1}(b_m\varphi(f_{k_{-1}}^{\pm\ell})b_m^{-1})$ does not belong to $\langle f_{k_{-1}}\rangle$, and using Theorem \ref{AI_bdrtori} (i) we get
	 \small $$a_{m-1}\varphi^{-1}(b_m\varphi(f_{k_{-1}}^{\pm\ell})b_m^{-1})a_{m-1}^{-1}\in\pi_1(X_{k_{-1}})\smallsetminus\pi_1(T^{-1})$$\normalsize We conclude that $|g\,c\,g^{-1}|\ge 3$ and thus $g\,c\,g^{-1}\not\in\pi_1(T)$.\\
	Assume now $a_m=1$. Consider $g\,c\,g^{-1}=(b_1\cdots b_m)c(b_1\cdots b_m)^{-1}$. The element $b_m c b_m^{-1}$ is still in $\pi_1(T^{+1})$, since $\pi_1(T^{+1})$ is normal in $\pi_1(X'')$. By Theorem \ref{AI_bdrtori} (i) we have $a_{m-1}\varphi^{-1}(b_m\,c\,b_m^{-1})a_{m-1}^{-1}\not\in\pi_1(T^{-1})$, unless $\varphi^{-1}(b_m cb_m^{-1})=f_{k_{-1}}^{\ell}$ for some $\ell\in\Z$ and  $a_{m-1}\in\pi_1(X_{k_{-1}})$. In the latter case 
	\small $$\varphi(a_{m-1}\varphi^{-1}(b_m\,c\,b_m^{-1})a_{m-1}^{-1})=\varphi(f_{k_{-1}}^{\pm\ell})$$
	 \normalsize
	 and by Lemma \ref{Gluingiso} we deduce that $\varphi(f_{k_{-1}}^{\pm\ell})$ does not belong to one of the two infinite cyclic subgroups generated by the regular fibers of $X''\simeq K\widetilde{\times}I$. Using Lemma \ref{kappatwistato} we deduce that $b_{m-1}\varphi(f_{k_{-1}}^{\pm\ell})b_{m-1}^{-1}\not\in\langle\varphi(f_{k_{-1}})\rangle$. Applying $\varphi^{-1}$  we obtain that $\varphi^{-1}(b_{m-1}\varphi(f_{k_{-1}}^{\pm\ell})b_{m-1}^{-1})\in\pi_1(T^{-1})\smallsetminus\langle f_{k_{-1}}\rangle$. By Theorem \ref{AI_bdrtori} (i) we conclude that
	 $a_{m-2}\varphi^{-1}(b_{m-1}\varphi(f_{k_{-1}}^{\pm\ell})b_{m-1}^{-1})a_{m-2}^{-1}\in\pi_1(X')\smallsetminus\pi_1(T^{-1})$.
	Hence, either $|gcg^{-1}|\ge 3$, or $|gcg^{-1}|=1$ and $gcg^{-1}\in\pi_1(X')\smallsetminus\pi_1(T^{-1})$.\\
	
	\textit{Proof of} (ND1). Let $X'$ be the closure of $X\smallsetminus T$.  Since $T^{-1}$ and $T^{+1}$ bound two hyperbolic JSJ-components (possibly the same JSJ-component) we know by Theorem 3 in \cite{delaHarpe2014onmalnormal} (or by Proposition \ref{int_leaves_stab}) that the subgroups $\pi_1(T^{-1})$ and $\pi_1(T^{+1})$ are malnormal in $\pi_1(X')$. Moreover, by Proposition \ref{int_leaves_stab} we know that they are  conjugately separated in $\pi_1(X')$. Consider $w\in\pi_1(T^{-\varepsilon_s})$ and any $|g|\ge 2$ we have that $|g\,w\,g^{-1}|\ge 2$. Let $g=w_0 t^{\varepsilon_1}\cdots t^{\varepsilon_s}w_s$ be a reduced form. We shall distinguish between the case where $\varepsilon_{s-1}=\varepsilon_s$ and the case where $\varepsilon_{s-1}=-\varepsilon_s$\\
	If $\varepsilon_{s-1}=\varepsilon_s$ it follows from the fact that $\pi_1(T^{\pm 1})$ are conjugately separated that $|gwg^{-1}|\ge 2$.
	Assume  that $\varepsilon_{s-1}=-\varepsilon_s$. Since we have chosen a reduced form for $g$, we  have $w_{s-1}\in\pi_1(X_{k_{\varepsilon_{s}}})\smallsetminus \pi_1(T^{\varepsilon_{s}})$. By \cite{delaHarpe2014onmalnormal} Theorem 3  (see also Proposition \ref{int_leaves_stab}) we conclude that  $w_{s-1}t^{\varepsilon_s}\,w\,t^{-\varepsilon_s}w_{s-1}^{-1}\not\in\pi_1(T^{\pm1})$  which shows that $$gwg^{-1}=(w_0t^{\varepsilon_1}\cdots t^{\varepsilon_{s-1}})(w_{s-1} t^{\varepsilon_s}w_s\,w\,w_s^{-1}t^{-\varepsilon_s}w_{s-1}^{-1})(w_0t^{\varepsilon_1}\cdots t^{\varepsilon_{s-1}})^{-1}$$  has length $|gwg^{-1}|\ge 2$, hence $gwg^{-1}\not\in\pi_1(T)$.\\
	
	 %If $w\in\pi_1(T_i^{-\varepsilon_1})$ we have that
	%\small
	%$$\gamma^{-1}w\gamma=(w_0t^{\varepsilon_1}\cdots t^{\varepsilon_s}w_s)^{-1}w(w_0t^{\varepsilon_1}\cdots t^{\varepsilon_s}w_s)$$
	%\normalsize
	%is a reduced form, unless $w_0\in\pi_1(T_i^{-\varepsilon_1})$ (by Theorem \ref{AI_bdrtori}). If this is the case we consider the element  
	%\small
	%$$(t^{\varepsilon_2}w_2\cdots t^{\varepsilon_s}w_s)^{-1}w_1^{-1}\varphi_i^{\varepsilon_1}(w_0^{-1}w w_0)w_1(t^{\varepsilon_2}w_2\cdots t^{\varepsilon_s}w_s)$$
	%\normalsize
	%and we consider two ca
	\textit{Proof of} (ND2). As in the previous case let $X'$ be the closure of $X\smallsetminus T$.  Let $w\in\pi_1(T)$ and let $g\in\pi_1(X)$ be such that $|g|\ge 3$. Consider a reduced form $g=w_0t^{\varepsilon_1}w_1\cdots t^{\varepsilon_s}w_s$. We conjugate $w$ by $g$ and we remark that the form 
$
	(w_0t^{\varepsilon_1}\cdots t^{\varepsilon_s})(w_s\,w\,w_s^{-1})	(w_0t^{\varepsilon_1}\cdots t^{\varepsilon_s})^{-1}
	$
	\normalsize
	is reduced unless $w_sww_s^{-1}\in\pi_1(T^{-\varepsilon_s})$. If the previous form is reduced then $gwg^{-1}\in\pi_1(X)\smallsetminus\pi_1(T)$. Otherwise, we need to distinguish  several cases, depending whether $\varepsilon_{s-2}$ and $\varepsilon_{s-1}$ have the same or the opposite sign with respect to $\varepsilon_{s}$.
	
	If $\varepsilon_{s-1}=-\varepsilon_s$ and $\varepsilon_{s-2}=\varepsilon_{s}$ we observe that, since $|g|\ge 3$ and the previous form is reduced, we have that $w_{s-1}\not\in\pi_1(T^{\varepsilon_s})$ and $w_{s-2}\not\in\pi_1(T^{-\varepsilon_s})$. Hence we can use Lemma \ref{tech} and conclude that $gwg^{-1}\in\pi_1(X)\smallsetminus\pi_1(T)$.
	
	Let $\varepsilon_{s-1}=\varepsilon_s$ and $\varepsilon_{s-2}=\varepsilon_s$.  Assume that $w_sww_s^{-1}\in\pi_1(T^{-\varepsilon_s})$ (otherwise $|gwg^{-1}|\ge 6$) and consider the element $\varphi^{\varepsilon_s}(w_sww_s^{-1})\in\pi_1(T^{\varepsilon_s})$. Thanks to Theorem \ref{AI_bdrtori} (ii) the element $w_{s-1}\varphi^{\varepsilon_s}(w_sww_s^{-1})w_{s-1}^{-1}$ belongs to $\pi_1(T^{-\varepsilon_s})$ if and only if $k=k_{-1}=k_{+1}$, the JSJ-component $X_k$ is Seifert fibered, the element $w_{s-1}\in\pi_1(X_k)$ and $\varphi^{\varepsilon_s}(w_sww_s^{-1})=f_{k}^\ell$ for a suitable $\ell\in\Z$
 (otherwise $|gwg^{-1}|\ge 4$). In this case
	 it follows from Lemma \ref{Gluingiso} that
	\small
	$$t^{\varepsilon_{s-1}} w_{s-1}\varphi^{\varepsilon_s}(w_sww_s^{-1})w_{s-1}^{-1} t^{-\varepsilon_{s-1}}=t^{\varepsilon_{s}} w_{s-1}\varphi^{\varepsilon_s}(w_sww_s^{-1})w_{s-1}^{-1} t^{-\varepsilon_{s}}=$$$$=t^{\varepsilon_{s}} f_k^{\pm\ell}t^{-\varepsilon_{s}}=\varphi^{\varepsilon_s}(f_k^{\pm \ell})\in\pi_1(T^{\varepsilon_{s}})\smallsetminus\langle f_k\rangle$$ 
	\normalsize
	where in the first equality we used the assumption $\varepsilon_{s-1}=\varepsilon_s$.
	  By Theorem \ref{AI_bdrtori} we see that
	  \small
	  $$ |(t^{\varepsilon_{s-2}}w_{s-2}t^{\varepsilon_{s-1}}w_{s-1} t^{\varepsilon_s}w_s)w (t^{\varepsilon_{s-2}}w_{s-2}t^{\varepsilon_{s-1}}w_{s-1} t^{\varepsilon_s}w_s)^{-1}|\ge 2$$ 
	  \normalsize
	  and we conclude that $|gwg^{-1}|\ge 2$ and thus $gwg^{-1}\not\in\pi_1(T)$.

	  If $\varepsilon_{s-1}=-\varepsilon_s$ and $\varepsilon_{s-2}=-\varepsilon_s$, we observe that $w_{s-1}\not\in\pi_1(T^{\varepsilon_s})$, since we have chosen a reduced form for $g$. The form for $gwg^{-1}$ is reduced unless  $w_sww_s^{-1}\in\pi_1(T^{-\varepsilon_s})$.  In this case, since $w_{s-1}\not\in\pi_1(T^{\varepsilon_s})$ we know by Theorem \ref{AI_bdrtori} (i) that $w_{s-1}t^{\varepsilon_s}w_s\,w\,w_s^{-1}t^{-\varepsilon_s}w_{s-1}^{-1}=w_{s-1}\varphi^{\varepsilon_s}(w_sww_s^{-1})w_{s-1}^{-1}\in\pi_1(T^{\varepsilon_s})$ if and only if $X_{k_{\varepsilon_s}}$ is a Seifert fibered JSJ-component, $w_{s-1}\in\pi_1(X_{k_{\varepsilon_s}})$ and $\varphi^{\varepsilon_s}(w_sww_s^{-1})=f_{k_{\varepsilon_s}}^{\ell}$ for a suitable $\ell\in\Z$. If one of the previous conditions fails $|gwg^{-1}|\ge 4$, and thus $gwg^{-1}\not\in\pi_1(T)$. Otherwise we have that \small$$t^{\varepsilon_{s-1}}(w_{s-1} f_{k_{\varepsilon_s}}^{\ell}w_{s-1}^{-1})t^{-\varepsilon_{s-1}}=t^{-\varepsilon_{s}}(w_{s-1} f_{k_{\varepsilon_s}}^{\ell}w_{s-1}^{-1})t^{\varepsilon_{s}}=\varphi^{-\varepsilon_s}(f_{k_{\varepsilon_s}}^{\pm\ell})\in\pi_1(T^{-\varepsilon_s})$$\normalsize  If $X_{k_{-\varepsilon_s}}$ is not Seifert fibered then it follows from Proposition \ref{int_leaves_stab} that $\pi_1(T^{\pm 1})$ are conjugately separated. Since $\varepsilon_{s-2}=\varepsilon_{s-1}=-\varepsilon_s$ we conclude that $|gwg^{-1}|\ge 2$. Assume that $X_{k_{-\varepsilon_s}}$ is Seifert fibered.  By Lemma \ref{Gluingiso} we have that \small$$t^{\varepsilon_{s-1}}(w_{s-1} f_{k_{\varepsilon_s}}^{\ell}w_{s-1}^{-1})t^{-\varepsilon_{s-1}}\in\pi_1(T^{-\varepsilon_s})\smallsetminus\langle f_{k_{-\varepsilon_s}}\rangle$$\normalsize and by Theorem \ref{AI_bdrtori} (ii) we conclude that $t^{\varepsilon_{s-1}}(w_{s-1} f_{k_{\varepsilon_s}}^{\ell}w_1)t^{-\varepsilon_{s-1}}\not\in\pi_1(T^{\varepsilon_s})$. It follows that $|gwg^{-1}|\ge 2$ and thus $gwg^{-1}\not\in\pi_1(T)$.
	 
	  Finally let $\varepsilon_{s-1}=\varepsilon_s$ and $\varepsilon_{s-2}=-\varepsilon_s$. Since the form $g=w_0t^{\varepsilon_1}w_1\cdots t^{\varepsilon_s}w_s$ is reduced, we deduce that  $w_{s-2}\in\pi_1(X')\smallsetminus\pi(T^{\varepsilon_s})$. Now take $w\in\pi_1(T)$; consider
	 \small
	 $$gwg^{-1}=(w_0t^{\varepsilon_1}\cdots t^{\varepsilon_s}w_s)\,w\, (w_0t^{\varepsilon_1}\cdots t^{\varepsilon_s}w_s)^{-1}$$
	 \normalsize
	The previous form is reduced unless $w_sww_s^{-1}\in\pi_1(T^{-\varepsilon_s})$. If this is the case, observe that $w_{s-1}t^{\varepsilon_{s}}(w_s\,w\,w_s) t^{-\varepsilon_s}w_{s-1}^{-1}=w_{s-1}\varphi^{\varepsilon_s}(w_sww_s^{-1})w_{s-1}^{-1}\in\pi_1(T^{-\varepsilon_s})$ if and only if $T^{\pm 1}$ are both boundary tori of the same Seifert fibered JSJ-component $X_{k}$,  $w_{s-1}\in\pi_1(X_k)$ and $\varphi^{\varepsilon_s}(w_sww_s^{-1})=f_k^\ell$ for a suitable $\ell\in\Z$. Now observe that 
	\small
	$$t^{\varepsilon_{s-1}}w_{s-1}\varphi^{\varepsilon_s}(w_sww_s^{-1})w_{s-1}t^{-\varepsilon_{s-1}}=t^{\varepsilon_{s}} f_k^{\pm\ell}t^{-\varepsilon_{s}}=\varphi^{\varepsilon_{s}}(f_k^{\pm\ell})\in\pi_1(T^{\varepsilon_s})\smallsetminus\langle f_k\rangle$$
	\normalsize
	where the last assertion follows from Lemma \ref{Gluingiso}.
	Since $w_{s-2}\in\pi_1(X')\smallsetminus\pi_1(T^{\varepsilon_s})$,  we deduce from Theorem \ref{AI_bdrtori} that $w_{s-2}\varphi^{\varepsilon_s}(f_k^{\pm\ell})w_{s-2}^{-1}\not\in\pi_1(T^{\pm 1})$. Thus $|gwg^{-1}|\ge 2$, which implies that $gwg^{-1}\in\pi_1(X)\smallsetminus\pi_1(T)$.
 	\end{proof}

\bibliographystyle{amsalpha}
\bibliography{biber_per_structure_short_final2}

\end{document}